\PassOptionsToPackage{numbers,sort&compress}{natbib}
\documentclass{svjour3}

\makeatletter
\def\cl@chapter{}%
\makeatother

\journalname{}
\smartqed

\usepackage{fix-cm}
\usepackage{graphicx}
\usepackage{amsmath,amssymb,amsfonts,mathtools}
\usepackage{microtype}
\usepackage{array,booktabs,tabularx}
\usepackage{flafter}
\usepackage{natbib}
\usepackage[colorlinks=true,linkcolor=blue,citecolor=blue,urlcolor=blue]{hyperref}
\usepackage[nameinlink,capitalize,noabbrev]{cleveref}
\renewcommand{\subclassname}{\textbf{Mathematics Subject Classification (2020)}\hfil\break}

\numberwithin{equation}{section}

\crefname{theorem}{Theorem}{Theorems}
\crefname{lemma}{Lemma}{Lemmas}
\crefname{proposition}{Proposition}{Propositions}
\crefname{corollary}{Corollary}{Corollaries}
\crefname{assumption}{Assumption}{Assumptions}
\crefname{definition}{Definition}{Definitions}
\crefname{remark}{Remark}{Remarks}
\crefname{equation}{equation}{equations}
\crefname{section}{Section}{Sections}
\crefname{appendix}{Appendix}{Appendices}
\newcolumntype{Y}{>{\raggedright\arraybackslash}X}

\DeclareMathOperator{\diag}{diag}
\DeclareMathOperator{\supp}{supp}
\DeclareMathOperator{\spec}{\sigma}
\newcommand{\R}{\mathbb R}
\newcommand{\eps}{\varepsilon}
\newcommand{\cH}{c_H}
\newcommand{\cA}{c_A}
\newcommand{\cP}{\mathcal P}
\newcommand{\norm}[1]{\lVert #1\rVert}
\newcommand{\ip}[2]{\langle #1,#2\rangle}
\newcommand{\ind}{\mathbf 1}
\newcommand{\doi}[1]{\href{https://doi.org/#1}{\nolinkurl{https://doi.org/#1}}}

\hypersetup{
  pdftitle={The Sharp Worst-Case Asymptotic Rate of the Barzilai--Borwein Method in Rd and Hilbert Spaces},
  pdfauthor={Shutai Yang and Ya-Xiang Yuan},
  pdfkeywords={Barzilai--Borwein method, sharp asymptotic rate, spectral measure, Hilbert space, nonlinear localization, ergodic optimization}
}

\begin{document}

\title{The Sharp Worst-Case Asymptotic Rate of the Barzilai--Borwein Method in $\R^d$ and Hilbert Spaces %
\thanks{The work of Shutai Yang was supported by the National Undergraduate Training Program for Innovation and Entrepreneurship (grant No.~202510358091).}}

\titlerunning{Sharp Worst-Case Asymptotic Rate of BB}

\author{Shutai Yang \and Ya-Xiang Yuan}
\authorrunning{S. Yang and Y.-X. Yuan}

\institute{Shutai Yang \at
    State Key Laboratory of Scientific and Engineering Computing, Academy of Mathematics and Systems Science, Chinese Academy of Sciences, and University of Chinese Academy of Sciences, Beijing, China \\
    School of Mathematical Sciences, University of Science and Technology of China, Hefei, China \\
    \email{taitai@mail.ustc.edu.cn}
    \and
    Ya-Xiang Yuan \at
    State Key Laboratory of Scientific and Engineering Computing, Academy of Mathematics and Systems Science, Chinese Academy of Sciences, Beijing, China \\
    \email{yyx@lsec.cc.ac.cn}
}

\date{Received: date / Accepted: date}
% The correct dates will be entered by the editor.

\maketitle

\begin{abstract}
We establish sharp asymptotic rates for the two Barzilai--Borwein (BB) rules on uniformly positive quadratics and local nonlinear problems.  In finite dimensions, for either fixed rule and an arbitrary positive first step, the gradient root factor is bounded by $(b_0-a_0)/(b_0+a_0)$, where $[a_0,b_0]$ is the initially active spectral interval.  Hence the worst trajectory factor is $c_H=(\kappa(H)-1)/(\kappa(H)+1)$.  When $H$ has at least two distinct eigenvalues, matched initialization and a balanced endpoint trajectory attain this value.  Under matched initialization, the same constant is the optimal uniform-envelope threshold.  For bounded, self-adjoint, uniformly positive operators on Hilbert space, scalar spectral measures yield the corresponding active-support bound and optimal matched uniform-envelope threshold, including continuous endpoint spectrum.  Finally, if the gradient is strictly Fr\'echet differentiable at a stationary point and its derivative is self-adjoint and uniformly positive, every $\gamma\in(c_*,1)$, where $c_*=(\kappa(A_*)-1)/(\kappa(A_*)+1)$, is a uniform local envelope rate for either pure BB rule.  Every well-defined trajectory converging to the stationary point has error and gradient root factors at most $c_*$ and objective-gap root factor at most $c_*^2$.  Over the class of objectives with prescribed distinct derivative endpoints $m_*<M_*$, matched endpoint trajectories for quadratic and $C^\infty$ genuinely nonquadratic examples in $\R^2$ attain these factors.
\end{abstract}

\section{Introduction}\label{sec:introduction}

Consider first the strongly convex quadratic problem
\begin{equation}\label{eq:quad-general}
  \min_{x\in\R^d} q(x)=\frac12 x^T Hx-b^Tx,
  \qquad H=H^T\succ0,
\end{equation}
with minimizer $x_*=H^{-1}b$ and gradients $g_k=H(x_k-x_*)$.  The BB1 iteration is
\begin{equation}\label{eq:BB1-general}
  x_{k+1}=x_k-\alpha_k g_k,
  \qquad \alpha_0>0,
  \qquad
  \alpha_k=\frac{g_{k-1}^Tg_{k-1}}{g_{k-1}^THg_{k-1}}
  \quad(k\ge1).
\end{equation}
The delayed Rayleigh quotient is the main structural feature.  For a current-gradient rule, spectral normalization gives a one-step transformation of a probability measure.  For BB1, the normalized state consists of two consecutive spectral-energy distributions.  This one-step memory prevents a direct reduction to the standard one-measure dynamics.

Barzilai and Borwein introduced the two classical secant choices and analyzed their two-dimensional behavior \cite{BarzilaiBorwein1988}.  Raydan proved global convergence of both choices on finite-dimensional strictly convex quadratics, and Dai and Liao established $R$-linear convergence in arbitrary dimension for BB1, with an analogous BB2 conclusion \cite{Raydan1993,DaiLiao2002}.  These works did not determine the worst trajectory root factor; Fletcher noted the absence of a realistic rate estimate \cite{Fletcher2005}.  Dai and Fletcher investigated a dimension-dependent transition between superlinear and linear behavior, while Dai derived the exact logarithmic recurrence in two dimensions and distinguished generic $R$-superlinear convergence from the equal-energy linear exception \cite{DaiFletcher2005,Dai2013}.

For matched BB1, Li and Sun \cite{LiSun2021} proved componentwise bounds with factor $1-1/\kappa(H)$ and constructed an endpoint-supported trajectory with exact factor
\[
  c_H=\frac{\kappa(H)-1}{\kappa(H)+1}.
\]
This leaves a gap between the upper and lower factors.  Li and Huang's Property~B gives componentwise $R$-linear bounds with factor $1-\lambda_{\min}/M_1$ for a broad class of retarded and hybrid rules, refined to $1-1/\kappa(H)$ when the reciprocal steps are bounded above by $\lambda_{\max}(H)$ \cite{LiHuang2025}.  For every $d\ge4$, Li, Jiang, and Hong used a computer-assisted attracting period-seven construction to exhibit open families of strictly convex quadratics with corresponding open sets of initial points on which BB1 converges without root-superlinear convergence \cite{LiJiangHong2026}.  We seek the universal sharp upper threshold.

Probability-measure dynamics for current-gradient methods go back to Akaike's endpoint analysis of exact steepest descent \cite{Akaike1959}.  Pronzato, Wynn, and Zhigljavsky extended the endpoint and two-cycle analysis to $P$-gradient rules in finite dimensions and Hilbert spaces and explicitly excluded BB from their class \cite{PronzatoWynnZhigljavsky2006}.  Under the reparameterization $P(\lambda)=\Psi(\lambda)/\lambda$, Huang, Dai, Liu, and Zhang obtained related finite-dimensional endpoint and asymptotic-ratio limits \cite{HuangDaiLiuZhang2022Asymptotic}.  The optimally tuned fixed-step method and exact steepest descent have the same worst asymptotic gradient-norm factor $c_H$ \cite{Akaike1959,Forsythe1968,Nocedal2002}.

Friedlander, Mart\'inez, Molina, and Raydan proved global convergence for a bounded-retard Rayleigh-quotient framework containing both BB choices; see \cite{FriedlanderMartinezMolinaRaydan1999,ZouMagoules2022}.  In Hilbert space, Azmi and Kunisch proved spectrally uniform $R$-linear convergence with continuous spectrum and iteration-dependent BB1/BB2 selection \cite{AzmiKunisch2020}.  We instead determine the sharp matched uniform-envelope threshold for each fixed rule.

For nonlinear problems, frozen-quadratic comparisons have yielded local results for pure, cyclic, and Hilbert-space BB methods under $C^3$ or locally Lipschitz-Hessian assumptions, sometimes conditional on trajectory convergence \cite{LiuDai2001,DaiLiao2002,DaiHagerSchittkowskiZhang2006,AzmiKunisch2022}.  We ask whether the sharp fixed-rule quadratic threshold persists under strict Fr\'echet differentiability of the gradient.

Our contributions are threefold.

First, in finite dimensions, if $a_0<b_0$ are the endpoints of the initially active spectrum, then, for either fixed BB rule and every $\alpha_0>0$, the gradient root factor is at most $(b_0-a_0)/(b_0+a_0)$.  Consequently, the supremum over the initial gradient and the positive first step equals $c_H$.  If $H$ has at least two distinct eigenvalues, a matched, balanced endpoint orbit attains this value.  Under matched initialization, $c_H$ is also the infimum of the uniform envelope rates: every $\gamma\in(c_H,1)$ is such a rate, whereas no $\gamma\in(0,c_H)$ is.  A fixed spectral conjugacy transfers these conclusions to BB2 and to fixed positive weighted delayed Rayleigh rules.

Second, for a bounded, self-adjoint, uniformly positive operator $A$ on a Hilbert space, scalar spectral measures give the same active-support bound for every positive first step.  Under matched initialization,
\[
  c_A=\frac{\max\spec(A)-\min\spec(A)}
            {\max\spec(A)+\min\spec(A)}
\]
is the optimal uniform-envelope threshold, even when one or both spectral endpoints lie only in the continuous spectrum.  If the two spectral endpoints are distinct eigenvalues, a matched, balanced endpoint vector attains this factor geometrically.

Third, let $A_*$ be the strict derivative of the gradient at a stationary point, and suppose that it is self-adjoint and uniformly positive.  Set
\[
  c_*:=\frac{\kappa(A_*)-1}{\kappa(A_*)+1}.
\]
For either fixed pure BB rule, every $\gamma\in(c_*,1)$ is a uniform local envelope rate.  Every well-defined trajectory converging to the stationary point has error and gradient root factors at most $c_*$ and objective-gap root factor at most $c_*^2$.  Over the class of objectives with prescribed distinct derivative endpoints $m_*<M_*$, these constants are sharp: matched endpoint trajectories for quadratic and $C^\infty$ genuinely nonquadratic examples in $\R^2$ attain them.

The finite-dimensional proof uses a two-step dynamics on pairs of probability distributions and a coboundary endpoint certificate.  Periodic-orbit bounds pass first to invariant measures through Poincar\'e recurrence and then to arbitrary trajectories through empirical measures.  Scalar spectral measures and shrinking endpoint bands replace coordinates in Hilbert space, while the nonlinear argument combines a uniform frozen-quadratic envelope with finite-horizon comparison.  For background on the coboundary viewpoint, see \cite{Jenkinson2019}.

\Crefrange{sec:finite-setting}{sec:unique} develop the finite-dimensional theory.  \Cref{sec:hilbert,app:hilbert-proofs} treat Hilbert spaces, and \cref{sec:nonlinear} proves the nonlinear result.

\section{Finite-dimensional setting and main results}\label{sec:finite-setting}

For a nonzero initial gradient, extend the gradient sequence by $g_k=0$ after finite termination and define
\begin{equation}\label{eq:rho-def}
  \rho_H(g_0,\alpha_0)
  :=\limsup_{k\to\infty}
  \left(\frac{\norm{g_k}}{\norm{g_0}}\right)^{1/k}.
\end{equation}
We use the matched initialization
\begin{equation}\label{eq:matched-init-BB1}
  \alpha_0^{\rm mat}=\frac{g_0^Tg_0}{g_0^THg_0}
\end{equation}
and, for BB1 under this matched initialization, call $\gamma\in(0,1)$ a \emph{uniform envelope rate} if there exists $C=C(H,\gamma)<\infty$ such that
\begin{equation}\label{eq:uniform-rate-definition}
  \norm{g_k}\le C\gamma^k\norm{g_0}
  \qquad\text{for every }g_0\ne0\text{ and every }k\ge0.
\end{equation}

For a function $\Psi:\spec(H)\to(0,\infty)$, the associated delayed weighted Rayleigh rule is
\begin{equation}\label{eq:weighted-delayed-intro}
  \alpha_k^{\Psi}
  =\frac{g_{k-1}^T\Psi(H)g_{k-1}}
         {g_{k-1}^T\Psi(H)Hg_{k-1}},
  \qquad k\ge1.
\end{equation}

Let $\{P_\lambda:\lambda\in\spec(H)\}$ be the orthogonal spectral projections and define the initially active spectrum
\begin{equation}\label{eq:active-spectrum-general}
  \Lambda(g_0)=\{\lambda\in\spec(H):P_\lambda g_0\ne0\}.
\end{equation}
If $\Lambda(g_0)$ contains at least two points, set
\begin{equation}\label{eq:active-c}
  a_0=\min\Lambda(g_0),\qquad
  b_0=\max\Lambda(g_0),\qquad
  c(g_0)=\frac{b_0-a_0}{b_0+a_0};
\end{equation}
if it is a singleton, set $c(g_0)=0$.  Finally, write
\begin{equation}\label{eq:cH-def}
  \cH=
  \frac{\lambda_{\max}(H)-\lambda_{\min}(H)}
       {\lambda_{\max}(H)+\lambda_{\min}(H)}
  =\frac{\kappa(H)-1}{\kappa(H)+1}.
\end{equation}

If $\spec(H)=\{\lambda\}$, then $H=\lambda I$.  An arbitrary positive first step either terminates immediately or leaves the gradient in the same eigenspace, after which the next BB1 or BB2 step is $1/\lambda$ and terminates.  Thus every trajectory stops in at most two updates and $\cH=0$.  All logarithmic normalized dynamics below are therefore stated only on supports containing at least two spectral points.

\begin{theorem}\label{thm:main}
Let the BB1 sequence be generated by \eqref{eq:BB1-general} with arbitrary $\alpha_0>0$.  Then
\begin{equation}\label{eq:active-main}
  \rho_H(g_0,\alpha_0)\le c(g_0)\le\cH
  \qquad(g_0\ne0).
\end{equation}
Moreover,
\begin{equation}\label{eq:global-main}
  \sup_{g_0\ne0,\,\alpha_0>0}\rho_H(g_0,\alpha_0)=\cH.
\end{equation}
If $H$ has at least two distinct eigenvalues, equality is attained under the matched initialization: take $g_0$ supported on the eigenspaces associated with
$a=\lambda_{\min}(H)$ and $b=\lambda_{\max}(H)$ and impose
\begin{equation}\label{eq:equal-endpoint-energy-intro}
  \norm{P_ag_0}^2=\norm{P_bg_0}^2.
\end{equation}
Then
\begin{equation}\label{eq:exact-geometric-intro}
  \norm{g_k}=\cH^k\norm{g_0}\qquad(k\ge0).
\end{equation}
The singleton-spectrum case gives both sides zero by finite termination.
\end{theorem}

\section{Spectral conjugacy and active-spectrum reduction}\label{sec:conjugacy}
The weighted delayed family \eqref{eq:weighted-delayed-intro} is spectrally conjugate to BB1.  It contains the positive-weight rules studied in \cite{HuangDaiLiuZhang2022Acceleration}, with BB1 and BB2 corresponding to $\Psi=I$ and $\Psi=H$.  The conjugacy below shows that a fixed positive spectral weight does not change the exact asymptotic root-factor problem.

\begin{lemma}\label{lem:spectral-conjugacy}
Let $\Psi:\spec(H)\to(0,\infty)$, let $g_0\ne0$, and set
\[
  S=\Psi(H)^{1/2},\qquad \widetilde g_k=Sg_k.
\]
Suppose that $\{g_k\}$ uses an arbitrary first step $\alpha_0>0$ and the delayed weighted rule \eqref{eq:weighted-delayed-intro} for $k\ge1$. Then
\begin{equation}\label{eq:conjugate-update}
\begin{aligned}
  \widetilde g_1
  &= (I-\alpha_0H)\widetilde g_0,
  &
  \widetilde g_{k+1}
  &= (I-\alpha_k^\Psi H)\widetilde g_k,
  &
  \alpha_k^\Psi
  &=\frac{\widetilde g_{k-1}^T\widetilde g_{k-1}}
          {\widetilde g_{k-1}^TH\widetilde g_{k-1}}
  &(k\ge1).
\end{aligned}
\end{equation}
Thus $\{\widetilde g_k\}$ is a BB1 gradient sequence for the same Hessian and with the same first step $\alpha_0$. Moreover,
\begin{equation}\label{eq:active-spectrum-conjugacy}
  \Lambda(\widetilde g_0)=\Lambda(g_0)
\end{equation}
and
\begin{equation}\label{eq:root-conjugacy}
  \limsup_{k\to\infty}
  \left(\frac{\norm{g_k}}{\norm{g_0}}\right)^{1/k}
  =
  \limsup_{k\to\infty}
  \left(\frac{\norm{\widetilde g_k}}{\norm{\widetilde g_0}}\right)^{1/k}.
\end{equation}
\end{lemma}

\begin{proof}
Because $S$ and $H$ are functions of the same symmetric matrix, they commute. Hence the first update satisfies
\[
  \widetilde g_1=S(I-\alpha_0H)g_0=(I-\alpha_0H)\widetilde g_0,
\]
and, for every $k\ge1$,
\[
  \widetilde g_{k+1}
  =S(I-\alpha_k^\Psi H)g_k
  =(I-\alpha_k^\Psi H)\widetilde g_k.
\]
The quotient in \eqref{eq:conjugate-update} follows from $S^2=\Psi(H)$. Positivity of $\Psi$ on the spectrum makes $S$ invertible and preserves every zero or nonzero spectral projection, proving \eqref{eq:active-spectrum-conjugacy}. Finally, with $\kappa_S=\norm{S}\norm{S^{-1}}$,
\[
  \kappa_S^{-1}
  \frac{\norm{g_k}}{\norm{g_0}}
  \le
  \frac{\norm{\widetilde g_k}}{\norm{\widetilde g_0}}
  \le
  \kappa_S
  \frac{\norm{g_k}}{\norm{g_0}}.
\]
The fixed equivalence constant disappears after taking $k$th roots.
\end{proof}

For a weighted sequence with $g_0\ne0$, write
\begin{equation}\label{eq:weighted-rho-def}
  \rho_{H,\Psi}(g_0,\alpha_0)
  :=\limsup_{k\to\infty}
  \left(\frac{\norm{g_k}}{\norm{g_0}}\right)^{1/k}.
\end{equation}

\begin{corollary}\label{cor:weighted-main}
Let $\Psi:\spec(H)\to(0,\infty)$, let the weighted delayed sequence use \eqref{eq:weighted-delayed-intro} for $k\ge1$, and let $\alpha_0>0$ be arbitrary. Then, for every $g_0\ne0$,
\[
  \rho_{H,\Psi}(g_0,\alpha_0)\le c(g_0)\le\cH,
  \qquad
  \sup_{g_0\ne0,\,\alpha_0>0}\rho_{H,\Psi}(g_0,\alpha_0)=\cH.
\]
If $H$ has at least two distinct eigenvalues, then under the matched weighted initialization
\begin{equation}\label{eq:matched-weighted-init}
  \alpha_0^{\Psi,\rm mat}
  =\frac{g_0^T\Psi(H)g_0}{g_0^T\Psi(H)Hg_0},
\end{equation}
the global factor is attained by a two-endpoint gradient satisfying
\begin{equation}\label{eq:weighted-balance}
  \Psi(a)\norm{P_ag_0}^2
  =\Psi(b)\norm{P_bg_0}^2,
  \qquad
  a=\lambda_{\min}(H),\quad b=\lambda_{\max}(H).
\end{equation}
In particular, $\Psi(t)=t$ gives BB2, for which the balance condition is
\[
  a\norm{P_ag_0}^2=b\norm{P_bg_0}^2.
\]
\end{corollary}

\begin{proof}
Apply \cref{thm:main} to the conjugate BB1 sequence from \cref{lem:spectral-conjugacy}. The matched initialization and balance condition become the ordinary BB1 matched initialization and equal endpoint energies for $\widetilde g_0$.
\end{proof}

Related renormalizations appear in current-gradient measure dynamics \cite{PronzatoWynnZhigljavsky2006,HuangDaiLiuZhang2022Asymptotic}.  Since the conjugacy reduces every fixed positive weighted delayed rule to BB1, we therefore analyze BB1 in the remainder of the finite-dimensional argument and transfer the conclusions through \cref{cor:weighted-main}.

We next reduce BB1 to distinct active eigenvalues.

\begin{lemma}\label{lem:active-reduction}
Every BB1 trajectory with $g_0\ne0$ reduces, without changing its asymptotic root factor, to a diagonal model
\begin{equation}\label{eq:diag-model}
  q(x)=\frac12 x^THx,
  \qquad
  H=\diag(\lambda_1,\ldots,\lambda_n),
  \qquad 0<\lambda_1<\cdots<\lambda_n,
\end{equation}
with one coordinate for each distinct active eigenspace.  Vanishing coordinates are removed by passage to invariant faces.
\end{lemma}

\begin{proof}
Translate $x_*$ to the origin and diagonalize $H$ orthogonally. Orthogonal transformations preserve Euclidean gradient norms and the BB1 Rayleigh quotient. On an eigenspace associated with $\lambda$, the update multiplies the entire projected gradient by the same scalar, so only its squared norm matters. The component recurrence is
\[
  P_{\lambda_i}g_{k+1}=(1-\alpha_k\lambda_i)P_{\lambda_i}g_k.
\]
Thus a zero component can never be recreated.  Only finitely many components can disappear.  After the last such event and at most one further state update, deleting the vanished coordinates gives an invariant lower-dimensional face and leaves all subsequent norms unchanged; discarding the preceding finite segment does not alter the limsup in \eqref{eq:rho-def}.
\end{proof}

\section{The normalized two-step dynamics}\label{sec:reduction}
This section gives the normalized reduction used below.  Fix an index set $I$ corresponding to at least two distinct active eigenvalues.
For every $J\subseteq I$ with $|J|\ge2$, write
\begin{equation}\label{eq:endpoint-notation}
\begin{gathered}
  i_-(J)=\min J,\qquad i_+(J)=\max J,\\
  a_J=\lambda_{i_-(J)},\qquad b_J=\lambda_{i_+(J)},\qquad
  m_J=\frac{a_J+b_J}{2},\qquad
  c_J=\frac{b_J-a_J}{b_J+a_J}.
\end{gathered}
\end{equation}
For $J=I$ we suppress the argument. Let $e_i$ denote the $i$th vertex of the probability simplex and define
\begin{equation}\label{eq:w-star-def}
  w_J^*=\frac12e_{i_-(J)}+\frac12e_{i_+(J)}.
\end{equation}
If $J\subseteq I$ and $|J|\ge2$, then $c_J\le c_I$, with equality only when $a_J=a_I$ and $b_J=b_I$.  Indeed, $(b-a)/(b+a)$ is strictly decreasing in $a$ and strictly increasing in $b$ for $0<a<b$, while $[a_J,b_J]\subseteq[a_I,b_I]$.

For $k\ge0$, define the squared spectral energies and their normalization by
\begin{equation}\label{eq:energy-def}
  E_i^k=(g_{k,i})^2,
  \qquad
  w_i^k=\frac{E_i^k}{\sum_{j\in I}E_j^k},
  \qquad w_k\in\Delta_I,
\end{equation}
whenever $g_k\ne0$, where
\[
  \Delta_I=\left\{w\in\R^I:w_i\ge0,\ \sum_{i\in I}w_i=1\right\}.
\]
Set
\begin{equation}\label{eq:u-def}
  u(w)=\sum_{i\in I}\lambda_iw_i
\end{equation}
and, for $z\in\Delta_I$,
\begin{equation}\label{eq:phi-def}
  \phi_i(z)=\left(1-\frac{\lambda_i}{u(z)}\right)^2.
\end{equation}
For $k\ge1$, the BB1 stepsize is $\alpha_k=1/u(w_{k-1})$, and therefore
\begin{equation}\label{eq:energy-rec}
  E_i^{k+1}=\phi_i(w_{k-1})E_i^k.
\end{equation}
Under the matched initialization \eqref{eq:matched-init-BB1}, the same relation holds for $k=0$ if $w_{-1}=w_0$.

Let
\begin{equation}\label{eq:X-state}
  X_I=\Delta_I\times\Delta_I.
\end{equation}
For $\chi=(w,z)\in X_I$, define
\begin{equation}\label{eq:r-def}
  r_I(w,z)^2=\sum_{i\in I}\phi_i(z)w_i
  =\ip{\phi(z)}{w}.
\end{equation}
Whenever $r_I(w,z)>0$, set
\begin{equation}\label{eq:T-def}
  T_I(w,z)=
  \left(
    \frac{\phi(z)\odot w}{\ip{\phi(z)}{w}},\,w
  \right),
\end{equation}
where $\odot$ denotes componentwise multiplication. If $g_k\ne0$ and $g_{k+1}\ne0$, then
\begin{equation}\label{eq:state-and-cocycle}
  \chi_k=(w_k,w_{k-1}),
  \qquad \chi_{k+1}=T_I\chi_k,
  \qquad \norm{g_{k+1}}=r_I(\chi_k)\norm{g_k}.
\end{equation}
Consequently, on every nonterminal segment for which the normalized states
$\chi_\ell,\ldots,\chi_{m-1}$ are defined,
\begin{equation}\label{eq:norm-cocycle}
  \frac{\norm{g_m}}{\norm{g_\ell}}
  =\prod_{k=\ell}^{m-1}r_I(\chi_k),
  \qquad
  \log\frac{\norm{g_m}}{\norm{g_\ell}}
  =\sum_{k=\ell}^{m-1}\log r_I(\chi_k).
\end{equation}
For matched initialization, $\chi_0=(w_0,w_0)$ and \eqref{eq:state-and-cocycle} starts at $k=0$. For arbitrary $\alpha_0>0$, it starts at $k=1$ from the unrestricted state $\chi_1=(w_1,w_0)\in X_I$. The function $r_I^2$ is continuous and bounded on $X_I$. If $r_I(\chi_k)=0$, then $g_{k+1}=0$ and the original method terminates.

\section{Sharp endpoint orbit and the two-dimensional BB recurrence}\label{sec:lower}
We first establish the matching lower bound and then record the projective recurrence that explains the instability of the extremal state.

\begin{proposition}\label{prop:endpoint-orbit}
Let $a<b$ be two eigenvalues. Suppose that $g_0\ne0$ is supported on the corresponding eigenspaces and satisfies
\begin{equation}\label{eq:endpoint-balance-BB1}
  \norm{P_ag_0}^2=\norm{P_bg_0}^2.
\end{equation}
Under the matched initialization,
\begin{equation}\label{eq:exact-endpoint-factor}
  \norm{g_k}=\left(\frac{b-a}{b+a}\right)^k\norm{g_0}
  \qquad(k\ge0).
\end{equation}
Equivalently, the normalized state is fixed and assigns mass $1/2$ to each of $a$ and $b$ in both coordinates.
\end{proposition}

\begin{proof}
The balance condition gives equal normalized endpoint energies, so $u(w_0)=(a+b)/2$. The matched stepsize is therefore $2/(a+b)$, and the endpoint multipliers are
\[
  1-\frac{2a}{a+b}=\frac{b-a}{b+a},
  \qquad
  1-\frac{2b}{a+b}=-\frac{b-a}{b+a}.
\]
With
$c_{a,b}=(b-a)/(b+a)$,
\begin{equation}\label{eq:endpoint-component-orbit}
  P_ag_k=c_{a,b}^kP_ag_0,
  \qquad
  P_bg_k=(-1)^kc_{a,b}^kP_bg_0.
\end{equation}
Thus the two endpoint components have equal squared energies at every step, while their signs alternate relative to one another, and the norm identity follows.
\end{proof}

Taking $a=\lambda_{\min}(H)$ and $b=\lambda_{\max}(H)$ proves the lower bound in \eqref{eq:global-main}.  This balanced orbit is the exceptional linear trajectory in the two-dimensional BB analysis \cite{BarzilaiBorwein1988} and appears in endpoint-supported form in \cite[Proposition~2]{LiSun2021}.  Under the conjugacy in \cref{lem:spectral-conjugacy}, equal endpoint energy becomes the weighted balance \eqref{eq:weighted-balance}.

For the two-eigenvalue model, write the gradient components as $\xi_k$ and $\eta_k$ and set
\begin{equation}\label{eq:projective-ratio-def}
  \zeta_k=\frac{\xi_k^2}{\eta_k^2}
\end{equation}
whenever both components are nonzero.

\begin{proposition}\label{prop:2d-instab}
As long as no component vanishes,
\begin{equation}\label{eq:qrec}
  \zeta_{k+1}=\frac{\zeta_k}{\zeta_{k-1}^2}\qquad(k\ge1).
\end{equation}
Under matched initialization, $\zeta_1=\zeta_0^{-1}$. The fixed solution $\zeta_k\equiv1$ is the unique matched orbit for which $\{\zeta_k\}$ remains in a compact subinterval of $(0,\infty)$. More precisely, if $\nu_k=\log\zeta_k$, then
\begin{equation}\label{eq:nu-rec}
  \nu_{k+1}=\nu_k-2\nu_{k-1},
  \qquad \nu_1=-\nu_0,
\end{equation}
and, unless $\nu_0=0$,
\begin{equation}\label{eq:nu-growth}
  \limsup_{k\to\infty}|\nu_k|^{1/k}=\sqrt2.
\end{equation}
\end{proposition}

\begin{proof}
If $r=\zeta_{k-1}$, then
\[
  u(w_{k-1})=\frac{ar+b}{r+1}.
\]
Consequently,
\[
  1-\frac{a}{u(w_{k-1})}=\frac{b-a}{ar+b},
  \qquad
  1-\frac{b}{u(w_{k-1})}=-\frac{(b-a)r}{ar+b}.
\]
Taking the ratio of the updated squared energies gives \eqref{eq:qrec}; the matched first step gives $\zeta_1=\zeta_0^{-1}$. Taking logarithms yields \eqref{eq:nu-rec}. Its characteristic roots are $(1\pm i\sqrt7)/2$, both of modulus $\sqrt2$. Every nonzero real solution therefore satisfies \eqref{eq:nu-growth}, while the matched constraint gives the zero solution exactly when $\nu_0=0$.
\end{proof}

Dai used the logarithmic recurrence above to separate the equal-energy linear orbit from the nondegenerate $R$-superlinear case and proved, in the latter case, that
\[
  \lim_{k\to\infty}\min\left\{
  \frac{\norm{g_{k+1}}}{\norm{g_k}},
  \frac{\norm{g_{k+2}}}{\norm{g_{k+1}}},
  \frac{\norm{g_{k+3}}}{\norm{g_{k+2}}}
  \right\}=0
\]
\cite[Sec.~2 and Theorem~2.1]{Dai2013}.  We later use this two-dimensional picture to calibrate the lower-bound mechanism.

\section{The endpoint inequality and a coboundary identity}\label{sec:endpoint}

The proof of the upper bound rests on the following elementary inequality.

\begin{lemma}\label{lem:endpoint-ineq}
Let $0<a<b$ and $u\in[a,b]$. Then
\begin{equation}\label{eq:endpoint-ineq}
  \left(\frac{u-a}{u}\right)^{b/(a+b)}
  \left(\frac{b-u}{u}\right)^{a/(a+b)}
  \le \frac{b-a}{b+a}.
\end{equation}
Equality holds if and only if $u=(a+b)/2$.
\end{lemma}

\begin{proof}
At $u=a$ or $u=b$, the left-hand side is zero, so the assertion is immediate. Suppose $u\in(a,b)$ and set
\[
  t=\frac{u-a}{b-u}>0.
\]
Then $u=(a+tb)/(1+t)$ and
\[
  \frac{u-a}{u}=\frac{t(b-a)}{a+tb},
  \qquad
  \frac{b-u}{u}=\frac{b-a}{a+tb}.
\]
The left-hand side of \eqref{eq:endpoint-ineq} is therefore
\[
  \frac{b-a}{a+tb}\,t^{b/(a+b)}.
\]
Weighted AM--GM, applied to $1$ and $t$ with weights $a/(a+b)$ and $b/(a+b)$, gives
\[
  \frac{a+tb}{a+b}\ge t^{b/(a+b)}.
\]
This proves the inequality. Equality in weighted AM--GM holds exactly when $t=1$, which is equivalent to $u=(a+b)/2$.
\end{proof}

For $J\subseteq I$ with $|J|\ge2$, let
\begin{equation}\label{eq:PJ-def}
  P_J=\left\{(w,z)\in\Delta_J\times\Delta_J:
  w_i>0,\ z_i>0\text{ for all }i\in J\right\},
\end{equation}
where $\Delta_J$ is identified with the corresponding face of $\Delta_I$.  To make the transfer term meaningful on boundary states reached after an interior coordinate vanishes, also set
\begin{equation}\label{eq:h-domain}
  \mathcal H_J=
  \left\{(w,z)\in\Delta_J\times\Delta_J:
  w_{i_-(J)}>0,\ w_{i_+(J)}>0\right\}.
\end{equation}
If $\chi=(w,z)\in P_J$, then $u(z)\in(a_J,b_J)$ and both $\chi$ and $T_I\chi$ belong to $\mathcal H_J$.  Define the endpoint transfer function on $\mathcal H_J$ by
\begin{equation}\label{eq:h-def}
  h_J(w,z)=
  -\frac{b_J}{2(a_J+b_J)}\log w_{i_-(J)}
  -\frac{a_J}{2(a_J+b_J)}\log w_{i_+(J)}
\end{equation}
and define on $P_J$ the nonnegative defect
\begin{equation}\label{eq:defect-def}
  \mathcal D_J(w,z)=\log c_J
  -\frac{b_J}{a_J+b_J}\log\frac{u(z)-a_J}{u(z)}
  -\frac{a_J}{a_J+b_J}\log\frac{b_J-u(z)}{u(z)}.
\end{equation}

\begin{lemma}\label{lem:cohomology}
For every $\chi\in P_J$,
\begin{equation}\label{eq:cohomology}
  \log r_I(\chi)
  =\log c_J-\mathcal D_J(\chi)+h_J(T_I\chi)-h_J(\chi).
\end{equation}
Moreover, $\mathcal D_J\ge0$, and $\mathcal D_J(\chi)=0$ if and only if $u(z)=m_J$.
\end{lemma}

\begin{proof}
The nonnegativity and the equality condition are the logarithmic form of \cref{lem:endpoint-ineq}. For the identity, abbreviate $a=a_J$, $b=b_J$, $u=u(z)$, and $R=r_I(w,z)^2$. If $T_I\chi=(w^+,w)$, then
\[
  w^+_{i_-(J)}=\frac{(1-a/u)^2w_{i_-(J)}}{R},
  \qquad
  w^+_{i_+(J)}=\frac{(1-b/u)^2w_{i_+(J)}}{R}.
\]
Since $a<u<b$,
\[
  |1-a/u|=\frac{u-a}{u},
  \qquad
  |1-b/u|=\frac{b-u}{u}.
\]
Substitution into \eqref{eq:h-def} gives
\begin{align*}
  h_J(T_I\chi)-h_J(\chi)
  &=-\frac{b}{a+b}\log\frac{u-a}{u}
    -\frac{a}{a+b}\log\frac{b-u}{u}
    +\frac12\log R\\
  &=-\frac{b}{a+b}\log\frac{u-a}{u}
    -\frac{a}{a+b}\log\frac{b-u}{u}
    +\log r_I(\chi).
\end{align*}
Rearranging and using \eqref{eq:defect-def} proves \eqref{eq:cohomology}.
\end{proof}

Equation \eqref{eq:cohomology} decomposes the one-step log growth rate into a critical slope $\log c_J$, a nonnegative defect $\mathcal D_J$, and a boundary term $h_J\circ T_I-h_J$ that telescopes along an orbit. All subsequent work extracts the sharp constant from this single identity.

\section{An elementary bound for exact periodic orbits}\label{sec:periodic}

Before the general argument we treat the most transparent case, where the coboundary telescopes over one period. 

Assume $|I|\ge2$.  An exact $p$-periodic orbit is a sequence
$\chi^{(0)},\ldots,\chi^{(p-1)}\in\{\chi\in X_I:r_I(\chi)>0\}$ such that
$\chi^{(j+1)}=T_I\chi^{(j)}$, with indices understood modulo $p$.  Write
\[
  \chi^{(j)}=(w^{(j)},z^{(j)}),
  \qquad
  \nu_j=u(z^{(j)}),
  \qquad
  R_j=r_I(\chi^{(j)})^2,
\]
and define the per-step norm factor
\begin{equation}\label{eq:periodic-factor}
  \rho=\left(\prod_{j=0}^{p-1}R_j\right)^{1/(2p)}.
\end{equation}

\begin{theorem}\label{thm:periodic}
Every exact periodic orbit satisfies
\[
  \rho\le c_I.
\]
Equality holds if and only if the orbit is the symmetric endpoint fixed point
\[
  \chi_I^*=(w_I^*,w_I^*),
  \qquad
  w_I^*=\frac12e_{i_-}+\frac12e_{i_+}.
\]
\end{theorem}

\begin{proof}
For every coordinate $i\in I$, the normalized update is
\begin{equation}\label{eq:periodic-coordinate-update}
  w_i^{(j+1)}
  =\frac{(1-\lambda_i/\nu_j)^2}{R_j}\,w_i^{(j)}.
\end{equation}
Let
\[
  S=\{i\in I:w_i^{(j)}>0\text{ for some }j\}
\]
be the active support of the periodic orbit.  A coordinate that vanishes can never be recreated.  Hence periodicity implies that every $i\in S$ is positive at every time, and no factor $1-\lambda_i/\nu_j$ with $i\in S$ can vanish.  Moreover, $S$ contains at least two indices.  Indeed, if $S=\{i\}$, then periodicity and the second-coordinate relation in $T_I$ give $w^{(j)}=z^{(j)}=e_i$ for every $j$, and hence $r_I(\chi^{(j)})=0$, contradicting the assumed regularity of the periodic orbit.  Multiplying \eqref{eq:periodic-coordinate-update} over one period and using $w_i^{(p)}=w_i^{(0)}>0$ gives
\begin{equation}\label{eq:periodic-mode}
  \left|\prod_{j=0}^{p-1}
  \left(1-\frac{\lambda_i}{\nu_j}\right)\right|
  =\rho^p,
  \qquad i\in S.
\end{equation}

Set
\[
  P(\lambda)=\prod_{j=0}^{p-1}\left(1-\frac{\lambda}{\nu_j}\right),
  \qquad
  \lambda_-:=\min_{i\in S}\lambda_i,
  \qquad
  \lambda_+:=\max_{i\in S}\lambda_i.
\]
Each $\nu_j$ is a convex combination of the eigenvalues indexed by $S$, so
\[
  a_I\le\lambda_-\le\nu_j\le\lambda_+\le b_I.
\]
On $[a_I,\lambda_-]$, every factor in $P(\lambda)$ is nonnegative and nonincreasing in $\lambda$.  Therefore $P$ is nonincreasing there, and \eqref{eq:periodic-mode} at the active lower endpoint gives
\begin{equation}\label{eq:periodic-F}
  \rho^p=P(\lambda_-)
  \le P(a_I)
  =\prod_{j=0}^{p-1}\frac{\nu_j-a_I}{\nu_j}
  =:F.
\end{equation}
On $[\lambda_+,b_I]$, every factor has nonpositive sign and
\[
  |P(\lambda)|=\prod_{j=0}^{p-1}\left(\frac{\lambda}{\nu_j}-1\right)
\]
is nondecreasing.  Hence \eqref{eq:periodic-mode} at the active upper endpoint yields
\begin{equation}\label{eq:periodic-G}
  \rho^p=|P(\lambda_+)|
  \le |P(b_I)|
  =\prod_{j=0}^{p-1}\frac{b_I-\nu_j}{\nu_j}
  =:G.
\end{equation}

Because every active coordinate is positive at every point of the periodic orbit, each mean $\nu_j$ lies strictly between the active endpoints $\lambda_-$ and $\lambda_+$.  Hence $F,G>0$.  Let $\vartheta=b_I/(a_I+b_I)$.  Then
\begin{align}
  \rho^p
  &\le\min\{F,G\}
   \le F^{\vartheta}G^{1-\vartheta}\notag\\
  &=\prod_{j=0}^{p-1}
    \left(\frac{\nu_j-a_I}{\nu_j}\right)^{b_I/(a_I+b_I)}
    \left(\frac{b_I-\nu_j}{\nu_j}\right)^{a_I/(a_I+b_I)}
  \le c_I^p,
  \label{eq:periodic-product-bound}
\end{align}
where the last inequality applies \cref{lem:endpoint-ineq} to each $\nu_j\in[a_I,b_I]$.  Thus $\rho\le c_I$.

Suppose now that $\rho=c_I$.  Then every inequality in \eqref{eq:periodic-product-bound} is an equality.  In particular, the product of $p$ quantities, each at most $c_I$, equals $c_I^p$; therefore equality holds in \cref{lem:endpoint-ineq} for every $j$.  Hence
\begin{equation}\label{eq:periodic-midpoint}
  \nu_j=m_I=\frac{a_I+b_I}{2}
  \qquad(j=0,\ldots,p-1).
\end{equation}
Substituting \eqref{eq:periodic-midpoint} into \eqref{eq:periodic-mode} gives, for every $i\in S$,
\[
  \left|1-\frac{\lambda_i}{m_I}\right|=c_I.
\]
The only solutions in $[a_I,b_I]$ are $\lambda_i=a_I$ and $\lambda_i=b_I$.  Thus $S\subseteq\{i_-,i_+\}$.  The mean condition \eqref{eq:periodic-midpoint} rules out a singleton support and forces equal endpoint weights.  Hence
\[
  z^{(j)}=w_I^*,
  \qquad
  w^{(j)}=z^{(j+1)}=w_I^*
\]
for every $j$, so the orbit is the fixed point $\chi_I^*$.  The converse follows immediately from \cref{prop:endpoint-orbit}.
\end{proof}

A general trajectory need not be periodic. The role of periodicity above --- returning to the starting state so that the boundary term $h_J\circ T_I-h_J$ telescopes to zero --- is played in the general case by Poincar\'e recurrence, to which we now turn.

\section{Compactification and invariant measures}\label{sec:invariant}
We now pass from exact periodic orbits to general invariant behavior.  Throughout this section, $|I|\ge2$. We adjoin a terminal state for finite termination. Let
\begin{equation}\label{eq:regular-singular}
  X_I^{\mathrm{reg}}=\{\chi\in X_I:r_I(\chi)>0\},
  \qquad
  \Sigma_I=X_I\setminus X_I^{\mathrm{reg}},
  \qquad
  \widehat X_I=X_I\sqcup\{\theta\}.
\end{equation}
Equip $\widehat X_I$ with the disjoint-union topology, with $\theta$ isolated. Then $\widehat X_I$ is a compact metric space.

Define
\begin{equation}\label{eq:hat-def}
  \widehat T_I\chi=
  \begin{cases}
    T_I\chi, & \chi\in X_I^{\mathrm{reg}},\\
    \theta, & \chi\in\Sigma_I\cup\{\theta\},
  \end{cases}
  \qquad
  \widehat f_I(\chi)=
  \begin{cases}
    \log r_I(\chi), & \chi\in X_I^{\mathrm{reg}},\\
    -\infty, & \chi\in\Sigma_I\cup\{\theta\}.
  \end{cases}
\end{equation}
The map $\widehat T_I$ is Borel and is continuous outside $\Sigma_I$. $\widehat f_I$ is upper semicontinuous: it is continuous on $X_I^{\mathrm{reg}}$, the point $\theta$ is isolated, and $\log r_I(\chi)\to-\infty$ as $\chi$ approaches $\Sigma_I$.  Throughout the paper, if an extended-real function $f$ is bounded above, we use the convention
\[
  \int f\,d\mu
  :=\lim_{M\to\infty}\int\max\{f,-M\}\,d\mu\in[-\infty,\infty).
\]
In particular, positive mass on a set where $f=-\infty$ forces the integral to be $-\infty$.  Since $|I|\ge2$,
\begin{equation}\label{eq:Mstar}
  M_I:=\max_{\chi\in\widehat X_I}\widehat f_I(\chi)
  =\frac12\log\max_{\chi\in X_I}r_I(\chi)^2
\end{equation}
is a finite real number.

For a nonempty $J\subseteq I$, set
\[
  X_J=\Delta_J\times\Delta_J,
  \qquad
  \widehat X_J=X_J\cup\{\theta\}.
\]
Every $\widehat X_J$ is forward invariant.  Indeed, if
$\chi=(w,z)\in X_J\cap X_I^{\mathrm{reg}}$, then $w_i=0$ for every
$i\notin J$, and the first component $w^+$ of $T_I\chi$ satisfies
\[
  w_i^+=\frac{\phi_i(z)w_i}{\ip{\phi(z)}{w}}=0
  \qquad(i\notin J).
\]
The second component of $T_I\chi$ is $w$, so $T_I\chi\in X_J$.  If
$\chi\in X_J\cap\Sigma_I$ or $\chi=\theta$, then
$\widehat T_I\chi=\theta\in\widehat X_J$.  Thus no spectral coordinate
outside $J$ can be created by the compactified dynamics.

We use the standard ergodic decomposition, Poincar\'e recurrence, and
Birkhoff's pointwise ergodic theorem.  See, for example, \cite{Walters1982}.  Since $\widehat X_I$ is a compact
metric space and $\widehat T_I$ is Borel, every invariant probability
measure admits an ergodic decomposition.  All integrals of
$\widehat f_I$ are understood in the extended sense fixed above.
Because $\widehat f_I\le M_I$, whenever
$\int\widehat f_I\,d\nu> -\infty$ its positive part is bounded and its
negative part is integrable.  Hence $\widehat f_I\in L^1(\nu)$, so
Birkhoff's theorem is applicable to $\widehat f_I$ under precisely the
finite-average hypothesis used below.

We shall use two standard weak-convergence facts.  If $F$ is upper semicontinuous, bounded above, and possibly takes the value $-\infty$, then
\begin{equation}\label{eq:extended-portmanteau}
  \mu_j\Rightarrow\mu
  \quad\Longrightarrow\quad
  \limsup_{j\to\infty}\int F\,d\mu_j\le\int F\,d\mu.
\end{equation}
Indeed, for $F_M=\max\{F,-M\}$ the bounded upper-semicontinuous Portmanteau inequality applies, and monotone convergence applied to $C-F_M$ permits $M\to\infty$, where $C$ is any finite upper bound for $F$.  We also use that if $G$ is bounded and Borel and its discontinuity set has $\mu$-measure zero, then $\int G\,d\mu_j\to\int G\,d\mu$.

\begin{lemma}\label{lem:avoid-sing}
Let $\nu$ be a $\widehat T_I$-invariant Borel probability measure. Then $\nu(\Sigma_I)=0$. If, in addition,
\[
  \int\widehat f_I\,d\nu>-\infty,
\]
then $\nu(\{\theta\})=0$ and therefore $\nu(\Sigma_I\cup\{\theta\})=0$.
\end{lemma}

\begin{proof}
By construction,
\[
  \widehat T_I^{-1}(\{\theta\})=\Sigma_I\cup\{\theta\}.
\]
Invariance gives
\[
  \nu(\{\theta\})=\nu(\Sigma_I)+\nu(\{\theta\}),
\]
so $\nu(\Sigma_I)=0$. If $\nu(\{\theta\})>0$, then $\widehat f_I(\theta)=-\infty$, together with the finite upper bound \eqref{eq:Mstar}, forces $\int\widehat f_I\,d\nu=-\infty$.
\end{proof}

\begin{lemma}\label{lem:support-face}
Let $\nu$ be an ergodic $\widehat T_I$-invariant Borel probability measure satisfying
\[
  \int\widehat f_I\,d\nu>-\infty.
\]
Then there exists $J\subseteq I$, with $|J|\ge2$, such that
\begin{equation}\label{eq:ergodic-face-full}
  \nu(P_J)=1.
\end{equation}
\end{lemma}

\begin{proof}
By \cref{lem:avoid-sing}, $\nu(\Sigma_I\cup\{\theta\})=0$.  Choose a minimal nonempty set $J\subseteq I$ such that $\nu(X_J)=1$; such a set exists because $I$ is finite.

If $|J|=1$, say $J=\{i\}$, then $X_J=\{(e_i,e_i)\}\subset\Sigma_I$, because $r_I(e_i,e_i)=0$.  This contradicts $\nu(X_J)=1$ and $\nu(\Sigma_I)=0$.  Hence $|J|\ge2$.

Fix $i\in J$ and set
\[
  A_i:=X_{J\setminus\{i\}}\cup\{\theta\}.
\]
This is exactly the compactified face on which the $i$th coordinate is
absent from both components, and the preceding face-invariance argument
shows that $A_i$ is forward invariant.  Hence
$A_i\subseteq\widehat T_I^{-1}A_i$.  Since $\nu$ is invariant,
\[
  \nu(\widehat T_I^{-1}A_i)=\nu(A_i),
\]
so the inclusion implies
$\nu(\widehat T_I^{-1}A_i\setminus A_i)=0$.  Thus $A_i$ is invariant
modulo $\nu$.  Ergodicity gives $\nu(A_i)\in\{0,1\}$, while minimality
of $J$ excludes the value $1$; consequently $\nu(A_i)=0$.

If $(w,z)\in X_J\cap X_I^{\mathrm{reg}}$ and $w_i=0$, then the $i$th coordinate is zero in both components of $\widehat T_I(w,z)=(w^+,w)$.  Hence
\[
  \{w_i=0\}\cap X_J\cap X_I^{\mathrm{reg}}
  \subseteq \widehat T_I^{-1}A_i.
\]
Using invariance, $\nu(A_i)=0$, $\nu(X_J)=1$, and $\nu(\Sigma_I\cup\{\theta\})=0$, we obtain
\[
  \nu\{w_i=0\}=0
  \qquad(i\in J).
\]

Extend the coordinate projections $\pi_1,\pi_2:X_I\to\Delta_I$ to $\widehat X_I$ by assigning them the same arbitrary value at $\theta$.  On $X_I^{\mathrm{reg}}$, the second output of $\widehat T_I(w,z)=T_I(w,z)$ is $w$, so
$\pi_2\circ\widehat T_I=\pi_1$ outside $\Sigma_I$.  Since $\nu(\Sigma_I\cup\{\theta\})=0$, invariance $(\widehat T_I)_*\nu=\nu$ therefore gives
\begin{equation}\label{eq:marginals-equal}
\begin{aligned}
  (\pi_2)_*\nu
  &=(\pi_2)_*(\widehat T_I)_*\nu
   =(\pi_2\circ\widehat T_I)_*\nu
   =(\pi_1)_*\nu.
\end{aligned}
\end{equation}
Thus the two coordinate marginals coincide.  In particular, for every
$i\in J$, let $Z_i=\{v\in\Delta_I:v_i=0\}$.  Then
\[
\begin{aligned}
  \nu\{z_i=0\}
  &=\bigl((\pi_2)_*\nu\bigr)(Z_i)
   =\bigl((\pi_1)_*\nu\bigr)(Z_i)
   =\nu\{w_i=0\}=0.
\end{aligned}
\]
Since $\nu(X_J)=1$, both coordinates are supported on $J$ and are
strictly positive there almost surely.  Intersecting these finitely many
full-measure positivity events gives $\nu(P_J)=1$.

\end{proof}

\begin{lemma}\label{lem:measure-bound}
If $\nu$ is a $\widehat T_I$-invariant Borel probability measure and
\[
  \int\widehat f_I\,d\nu>-\infty,
\]
then
\begin{equation}\label{eq:measure-bound}
  \int\widehat f_I\,d\nu\le\log c_I.
\end{equation}
\end{lemma}

\begin{proof}
Let $\nu=\int\eta\,d\Lambda(\eta)$ be the ergodic decomposition.  Applying the decomposition to the bounded truncations $\max\{\widehat f_I,-M\}$ and then letting $M\to\infty$ gives
\[
  \int\widehat f_I\,d\nu
  =\int\left(\int\widehat f_I\,d\eta\right)d\Lambda(\eta).
\]
Since the left-hand side is finite and $\widehat f_I$ is bounded above, $\int\widehat f_I\,d\eta>-\infty$ for $\Lambda$-almost every ergodic component.  It therefore suffices to prove the bound for an ergodic $\nu$ with finite integral.

By \cref{lem:support-face}, $\nu(P_J)=1$ for some $J\subseteq I$
with $|J|\ge2$.  Set
\begin{equation}\label{eq:persistent-face-core}
  P_J^\infty
  :=\bigcap_{n\ge0}
    \widehat T_I^{-n}\bigl(P_J\cap X_I^{\mathrm{reg}}\bigr).
\end{equation}
Because $\nu(P_J)=1$ and $\nu(\Sigma_I)=0$, the set
$P_J\cap X_I^{\mathrm{reg}}$ has full measure.  Invariance then gives,
for every $n\ge0$,
\[
  \nu\!\left(
    \widehat T_I^{-n}(P_J\cap X_I^{\mathrm{reg}})
  \right)
  =\nu(P_J\cap X_I^{\mathrm{reg}})=1.
\]
The countable intersection in \eqref{eq:persistent-face-core} therefore
also has full measure:
\[
  \nu(P_J^\infty)=1.
\]
For every $\chi\in P_J^\infty$ and every $k\ge0$,
\[
  \widehat T_I^k\chi\in P_J\cap X_I^{\mathrm{reg}}.
\]
Thus the entire forward orbit remains nonterminal in the same relative
interior face.  In particular, the coboundary identity
\eqref{eq:cohomology} applies at every iterate and 
$h_J$ is finite along the whole orbit.

Since $\widehat f_I\in L^1(\nu)$, both the Birkhoff-generic set for
$\widehat f_I$ and the set of recurrent points have full measure.
Intersect these two sets with $P_J^\infty$ and choose $\chi$ in the
resulting full-measure intersection.  Then there are $n_\ell\to\infty$
such that $\widehat T_I^{n_\ell}\chi\to\chi$, and
\[
  \frac1n\sum_{k=0}^{n-1}\widehat f_I(\widehat T_I^k\chi)
  \longrightarrow \int\widehat f_I\,d\nu.
\]
Since $\chi\in P_J$ and $P_J$ is open in its face, the recurrence
$\widehat T_I^{n_\ell}\chi\to\chi$ takes place inside the domain on which
$h_J$ is continuous.  Hence
\[
  h_J(\widehat T_I^{n_\ell}\chi)\longrightarrow h_J(\chi),
\]
and therefore,
\[
  \frac{h_J(\widehat T_I^{n_\ell}\chi)-h_J(\chi)}{n_\ell}\longrightarrow0.
\]
Summing \eqref{eq:cohomology} from $k=0$ to $n_\ell-1$ yields
\begin{align*}
  \frac1{n_\ell}\sum_{k=0}^{n_\ell-1}
  \widehat f_I(\widehat T_I^k\chi)
  &=\log c_J
    -\frac1{n_\ell}\sum_{k=0}^{n_\ell-1}
      \mathcal D_J(\widehat T_I^k\chi)+
    \frac{h_J(\widehat T_I^{n_\ell}\chi)-h_J(\chi)}{n_\ell}.
\end{align*}
Letting $\ell\to\infty$ and using $\mathcal D_J\ge0$ gives
\[
  \int\widehat f_I\,d\nu\le\log c_J\le\log c_I.
\]
The same bound holds for almost every ergodic component and hence for the original invariant measure.
\end{proof}

\section{Empirical measures and the sharp upper bound}\label{sec:empirical}

The following is an extended-real, Borel-map variant of the semi-uniform ergodic principle developed by Sturman and Stark \cite{SturmanStark2000}.  We include the short proof because the maps used below may be discontinuous on $D$, while the upper-semicontinuous potential takes the value $-\infty$ there.

\begin{proposition}\label{prop:abstract-semiuniform}
Let $X$ be a compact metric space, let $T:X\to X$ be Borel, and let $D\subset X$ be a Borel set containing every discontinuity point of $T$.  Let $f:X\to[-\infty,\infty)$ be upper semicontinuous and bounded above, assume that $f=-\infty$ on $D$, and let $\beta\in\R$.  Suppose that
\begin{equation}\label{eq:abstract-invariant-bound}
  \int f\,d\mu\le\beta
\end{equation}
for every $T$-invariant Borel probability measure $\mu$ satisfying $\int f\,d\mu>-\infty$.  Then, for every $\eta>0$, there exists $B_\eta<\infty$ such that
\begin{equation}\label{eq:abstract-semiuniform}
  \sum_{k=0}^{n-1}f(T^kx)
  \le n(\beta+\eta)+B_\eta
\end{equation}
for every integer $n\ge1$ and every $x\in X$ such that $x,Tx,\ldots,T^{n-1}x$ all lie outside $D$.
\end{proposition}

\begin{proof}
Suppose that \eqref{eq:abstract-semiuniform} fails for some $\eta>0$.  For every $j$ choose an orbit segment $(x_j,n_j)$ avoiding $D$ such that
\begin{equation}\label{eq:abstract-unbounded-excess}
  \sum_{k=0}^{n_j-1}f(T^kx_j)-n_j(\beta+\eta)>j.
\end{equation}
Since $f$ is bounded above, $n_j\to\infty$.  Passing to a subsequence, the empirical measures
\[
  \mu_j=\frac1{n_j}\sum_{k=0}^{n_j-1}\delta_{T^kx_j}
\]
converge weakly to a probability measure $\mu$.

For $M>0$, set $f_M=\max\{f,-M\}$.  Then $f_M$ is bounded and upper semicontinuous, and $f_M\ge f$.  From \eqref{eq:abstract-unbounded-excess} and \eqref{eq:extended-portmanteau},
\[
  \int f_M\,d\mu
  \ge\limsup_{j\to\infty}\int f_M\,d\mu_j
  \ge\beta+\eta.
\]
Let $U=\max\{0,\sup_X f\}$.  Since $U-f_M\uparrow U-f$, monotone convergence gives
\begin{equation}\label{eq:abstract-high-limit}
  \int f\,d\mu=\lim_{M\to\infty}\int f_M\,d\mu
  \ge\beta+\eta>-\infty.
\end{equation}
Because $f=-\infty$ on $D$ and is bounded above, \eqref{eq:abstract-high-limit} implies $\mu(D)=0$.

Let $\varphi\in C(X)$.  Every discontinuity of $\varphi\circ T$ is a discontinuity of $T$ and therefore lies in $D$.  Since $\mu(D)=0$, the function $\varphi\circ T$ is bounded and $\mu$-almost everywhere continuous.  The Portmanteau consequence for bounded almost-everywhere continuous functions therefore gives
\[
  \int\varphi\circ T\,d\mu_j\longrightarrow
  \int\varphi\circ T\,d\mu.
\]
Weak convergence also gives $\int\varphi\,d\mu_j\to\int\varphi\,d\mu$, while
\[
  \left|\int\varphi\circ T\,d\mu_j-
  \int\varphi\,d\mu_j\right|
  =\frac{|\varphi(T^{n_j}x_j)-\varphi(x_j)|}{n_j}
  \le\frac{2\norm{\varphi}_\infty}{n_j}\longrightarrow0.
\]
Thus $\mu$ is $T$-invariant, contradicting \eqref{eq:abstract-invariant-bound} and \eqref{eq:abstract-high-limit}.
\end{proof}

Call
\[
  \chi,\widehat T_I\chi,\ldots,\widehat T_I^{m-1}\chi
\]
a \emph{nonterminal orbit segment} if all $m$ displayed points lie in $X_I^{\mathrm{reg}}$.

\begin{lemma}\label{lem:uniform-orbit}
For every $\eta>0$, there exists $B_{I,\eta}<\infty$ such that every nonterminal orbit segment with arbitrary initial state $\chi\in X_I$ satisfies
\begin{equation}\label{eq:uniform-orbit}
  \sum_{k=0}^{m-1}\widehat f_I(\widehat T_I^k\chi)
  \le m(\log c_I+\eta)+B_{I,\eta}
  \qquad(m\ge1).
\end{equation}
\end{lemma}

\begin{proof}
Apply \cref{prop:abstract-semiuniform} with
\[
  X=\widehat X_I,\qquad T=\widehat T_I,\qquad
  D=\Sigma_I\cup\{\theta\},\qquad
  f=\widehat f_I,\qquad \beta=\log c_I.
\]
The preceding properties of $\widehat T_I$ and $\widehat f_I$, together with \cref{lem:measure-bound}, verify all the hypotheses; the nonterminal segments are precisely the admissible segments.
\end{proof}

\begin{proof}[Proof of \cref{thm:main}]
Apply \cref{lem:active-reduction} and let $I_0$ index the distinct eigenvalues in $\Lambda(g_0)$. If $|I_0|=1$, then $g_1$ is either zero or remains in the same eigenspace; in the latter case $\alpha_1$ is the reciprocal eigenvalue and $g_2=0$. Hence the root factor is zero.

Assume $|I_0|\ge2$. If $g_1=0$, the conclusion is immediate. Otherwise, the normalized state
\[
  \chi_1=(w_1,w_0)
\]
belongs to the closed face $X_{I_0}$. For every $m\ge2$ with $g_m\ne0$, the states $\chi_1,\ldots,\chi_{m-1}$ form a nonterminal orbit segment, and
\begin{equation}\label{eq:arbitrary-init-product}
  \frac{\norm{g_m}}{\norm{g_0}}
  =\frac{\norm{g_1}}{\norm{g_0}}
   \prod_{k=1}^{m-1}r_{I_0}(\chi_k),
  \qquad \chi_k=\widehat T_{I_0}^{k-1}\chi_1.
\end{equation}
Applying \cref{lem:uniform-orbit} to this segment gives, for every $\eta>0$,
\begin{equation}\label{eq:arbitrary-init-bound}
  \norm{g_m}
  \le \norm{g_1}\exp(B_{I_0,\eta})
  (c_{I_0}e^\eta)^{m-1}
  \qquad(m\ge2).
\end{equation}
If termination occurs at or before index $m$, then $g_m=0$ and the same inequality is trivial. The factor $\norm{g_1}/\norm{g_0}$ and the shift from the exponent $m$ to $m-1$ concern only the single initial step.  Both disappear after taking $m$th roots.  Thus, after dividing by $\norm{g_0}$ and letting $m\to\infty$, we obtain
\[
  \rho_H(g_0,\alpha_0)\le c_{I_0}e^\eta.
\]
Letting $\eta\downarrow0$ proves the active-spectrum bound. Since the interval determined by $I_0$ is contained in $[\lambda_{\min}(H),\lambda_{\max}(H)]$, we have $c_{I_0}\le\cH$.

The exact endpoint orbit in \cref{prop:endpoint-orbit}, with the global spectral endpoints and matched initialization, has root factor $\cH$. This proves \eqref{eq:global-main} and the attainment statement. The one-eigenvalue case was handled above.
\end{proof}

\begin{corollary}\label{cor:uniform-envelope}
Assume \eqref{eq:matched-init-BB1}. For every $\eps$ satisfying
\[
  0<\eps<1-\cH,
\]
there exists $C_\eps(H)<\infty$, independent of $g_0$, such that
\begin{equation}\label{eq:uniform-envelope}
  \norm{g_k}
  \le C_\eps(H)(\cH+\eps)^k\norm{g_0}
  \qquad(k\ge0).
\end{equation}
Consequently, the infimum of the uniform envelope rates in \eqref{eq:uniform-rate-definition} is exactly $\cH$.
\end{corollary}

\begin{proof}
Under matched initialization, the normalized orbit starts from $\chi_0=(w_0,w_0)$. There are only finitely many possible active spectra $J$ drawn from the distinct eigenvalues of $H$. If $|J|\ge2$, choose $\eta_J>0$ such that
\[
  c_Je^{\eta_J}\le\cH+\eps.
\]
The cocycle identity and \cref{lem:uniform-orbit} give
\[
  \norm{g_k}
  \le \exp(B_{J,\eta_J})(\cH+\eps)^k\norm{g_0}
\]
for every matched trajectory with active spectrum $J$. If $|J|=1$, the first step terminates exactly. Taking the maximum of $1$ and the finitely many constants $\exp(B_{J,\eta_J})$ proves \eqref{eq:uniform-envelope}. Conversely, if $\cH>0$, the exact endpoint orbit excludes every $\gamma\in(0,\cH)$.
\end{proof}

\begin{corollary}\label{cor:weighted-envelope}
Let $\Psi:\spec(H)\to(0,\infty)$ and use the matched weighted initialization \eqref{eq:matched-weighted-init}. For every $0<\eps<1-\cH$, there is $C_\eps(H,\Psi)<\infty$ such that
\[
  \norm{g_k}\le C_\eps(H,\Psi)(\cH+\eps)^k\norm{g_0}
  \qquad(g_0\ne0,\ k\ge0).
\]
Consequently, the infimum of $\gamma\in(0,1)$ for which the displayed estimate holds with a constant independent of $g_0$ is exactly $\cH$.
\end{corollary}

\begin{proof}
Apply \cref{cor:uniform-envelope} to $\widetilde g_k=\Psi(H)^{1/2}g_k$ and use norm equivalence as in \cref{lem:spectral-conjugacy}.
If $H$ has only one distinct eigenvalue, the threshold is zero by finite termination.  Otherwise, the exact matched weighted endpoint orbit in \cref{cor:weighted-main} excludes every $\gamma\in(0,\cH)$.
\end{proof}

\begin{corollary}\label{cor:error-objective}
Assume $g_0\ne0$, let $e_k=x_k-x_*$, and set $\Delta_k=q(x_k)-q(x_*)$. For BB1, and also for every fixed positive weighted rule, let the first step $\alpha_0>0$ be arbitrary and let $\rho$ denote the root factor of the actual gradient sequence $\{g_k\}$; by \cref{lem:spectral-conjugacy}, this is also the root factor of the conjugate sequence. Then
\begin{equation}\label{eq:error-root-equality}
  \limsup_{k\to\infty}
  \left(\frac{\norm{e_k}}{\norm{e_0}}\right)^{1/k}
  =\rho
\end{equation}
and
\begin{equation}\label{eq:objective-root-square}
  \limsup_{k\to\infty}
  \left(\frac{\Delta_k}{\Delta_0}\right)^{1/k}
  =\rho^2.
\end{equation}
Therefore, over arbitrary positive first steps, the sharp global worst-case factors are $\cH$ for the error norm and $\cH^2$ for the objective gap.  When $H$ has at least two distinct eigenvalues, these factors are attained by the corresponding matched balanced endpoint orbit.
\end{corollary}

\begin{proof}
Since $g_k=He_k$,
\[
  \lambda_{\min}(H)\norm{e_k}
  \le\norm{g_k}
  \le\lambda_{\max}(H)\norm{e_k}.
\]
Fixed norm-equivalence constants disappear after taking $k$th roots, proving \eqref{eq:error-root-equality}. Moreover,
\[
  \Delta_k=\frac12e_k^THe_k
  =\frac12g_k^TH^{-1}g_k,
\]
so
\[
  \frac{1}{2\lambda_{\max}(H)}\norm{g_k}^2
  \le\Delta_k
  \le\frac{1}{2\lambda_{\min}(H)}\norm{g_k}^2.
\]
This proves \eqref{eq:objective-root-square}. Sharpness follows from the endpoint orbit; for a weighted method use \eqref{eq:weighted-balance}.
\end{proof}

\section{Uniqueness of the maximizing invariant measure}\label{sec:unique}
The coboundary certificate also determines the unique invariant measure attaining the sharp average. Recall
\[
  \chi_I^*=(w_I^*,w_I^*),
  \qquad
  w_I^*=\frac12e_{i_-}+\frac12e_{i_+}.
\]

\begin{proposition}\label{prop:unique}
Let $|I|\ge2$. Then
\begin{equation}\label{eq:max-measure-value}
  \sup\left\{
    \int\widehat f_I\,d\nu:
    \nu\text{ is a }\widehat T_I\text{-invariant probability measure}
  \right\}
  =\log c_I.
\end{equation}
The unique maximizing invariant probability measure is
\begin{equation}\label{eq:unique-max-measure}
  \delta_{\chi_I^*}.
\end{equation}
\end{proposition}

\begin{proof}
The upper bound is \cref{lem:measure-bound}.  The point $\chi_I^*$ is a fixed point of $\widehat T_I$: $u(z)=(a_I+b_I)/2$, so $\phi_{i_-}(z)=\phi_{i_+}(z)=c_I^2$, the two endpoint weights $1/2$ are preserved by $T_I$, and $\widehat f_I(\chi_I^*)=\log c_I$.  Hence $\delta_{\chi_I^*}$ is invariant and attains the supremum.

For uniqueness, let $\nu$ be an ergodic maximizing measure, $\int\widehat f_I\,d\nu=\log c_I$, and adopt the notation of \cref{lem:measure-bound}.  Choose $\chi\in P_J^\infty$ so that $\chi$ is recurrent and belongs to the countable intersection of the Birkhoff-generic sets for $\widehat f_I$ and for all bounded truncations
\[
  \mathcal D_{J,M}:=\min\{\mathcal D_J,M\},
  \qquad M\in\mathbb N,
\]
where the truncations are extended by zero outside $P_J$.  Let $n_\ell\to\infty$ be recurrence times for $\chi$.  Equality throughout the proof of \cref{lem:measure-bound} then forces the following.

\emph{(1) $c_J=c_I$.}  Since $c_J\le c_I$ with equality only when $a_J=a_I$ and $b_J=b_I$, the face $J$ contains both endpoints of $I$.

\emph{(2) $\mathcal D_J=0$ $\nu$-almost everywhere.}  The chain in \cref{lem:measure-bound} shows
\[
  \frac1{n_\ell}\sum_{k=0}^{n_\ell-1}
  \mathcal D_J(\widehat T_I^k\chi)\longrightarrow0
\]
along the recurrent generic point.  For each $M>0$,
$0\le\mathcal D_{J,M}\le\mathcal D_J$, so Birkhoff's theorem gives
\[
  \int\mathcal D_{J,M}\,d\nu=0.
\]
Letting $M\to\infty$ and using monotone convergence yields
$\int\mathcal D_J\,d\nu=0$, hence $\mathcal D_J=0$ $\nu$-almost everywhere.  By the equality case of \cref{lem:cohomology},
\[
  u(z)=m_J=\frac{a_J+b_J}{2}
       =\frac{a_I+b_I}{2}
  \qquad \nu\text{-almost everywhere}.
\]
Since the two coordinate marginals coincide, also $u(w)=m_J$ $\nu$-almost everywhere.

\emph{(3) No mass on interior coordinates.}  Put
\[
  a=a_I,\qquad b=b_I,\qquad m=\frac{a+b}{2}.
\]
Then
\[
  \left(1-\frac{a}{m}\right)^2
  =\left(1-\frac{b}{m}\right)^2=c_I^2,
\]
while, if $J$ contains an intermediate eigenvalue, the number
\[
  \rho_*:=
  \max_{\substack{i\in J\\a<\lambda_i<b}}
  \left(1-\frac{\lambda_i}{m}\right)^2
\]
satisfies $\rho_*<c_I^2$.  Hence, whenever $u(z)=m$,
\[
  r_I(w,z)^2
  \le
  c_I^2\bigl(w_{i_-}+w_{i_+}\bigr)
  +\rho_*\!\!\sum_{\substack{i\in J\\a<\lambda_i<b}}\!\!w_i
  \le c_I^2,
\]
and the last inequality is strict whenever an interior weight is positive.  If $J$ contains no intermediate eigenvalue, this conclusion is immediate without the term involving $\rho_*$.  Since $\widehat f_I=\log r_I\le\log c_I$ $\nu$-almost everywhere and
$\int\widehat f_I\,d\nu=\log c_I$, we must have $r_I=c_I$ $\nu$-almost everywhere.  Thus all interior weights vanish.  Combined with the positivity $w_i>0$ and $z_i>0$ for $i\in J$ supplied by \cref{lem:support-face}, this gives
\[
  J=\{i_-,i_+\}.
\]
The mean conditions $u(w)=u(z)=m$ then give endpoint weights $1/2$ in both $w$ and $z$.  Thus $\nu=\delta_{\chi_I^*}$ for every ergodic maximizing measure.

Finally, let $\nu$ be any invariant measure with
$\int\widehat f_I\,d\nu=\log c_I$.  In its ergodic decomposition, almost every component $\eta$ satisfies
$\int\widehat f_I\,d\eta\le\log c_I$ by \cref{lem:measure-bound}, while the average of these integrals equals $\log c_I$.  Hence almost every component is an ergodic maximizing measure and therefore equals $\delta_{\chi_I^*}$, giving
$\nu=\delta_{\chi_I^*}$.
\end{proof}

\section{Extension to bounded positive operators on Hilbert spaces}\label{sec:hilbert}

We now extend the finite-dimensional theory from discrete spectral weights to scalar spectral measures.

\subsection{Statement of the extension}\label{subsec:hilb-statement}

Let $\mathcal H$ be a real Hilbert space and let
\begin{equation}\label{eq:hilb-operator-bounds}
  A=A^*\in\mathcal L(\mathcal H),
  \qquad 0<mI\le A\le MI<\infty,
\end{equation}
where
\[
  m=\min\spec(A),\qquad M=\max\spec(A).
\]
For
\begin{equation}\label{eq:hilb-quadratic}
  Q(x)=\frac12\ip{Ax}{x}-\ip{b}{x},
\end{equation}
the minimizer is $x_*=A^{-1}b$.  Set $g_k=A(x_k-x_*)$ and apply BB1 in the form
\begin{equation}\label{eq:hilb-BB1}
  x_{k+1}=x_k-\alpha_k g_k,
  \qquad \alpha_0>0,
  \qquad
  \alpha_k=\frac{\norm{g_{k-1}}^2}{\ip{Ag_{k-1}}{g_{k-1}}}
  \quad(k\ge1).
\end{equation}
As before, the gradient sequence is set equal to zero after finite termination.  For $g_0\ne0$, define
\begin{equation}\label{eq:hilb-rho}
  \rho_A(g_0,\alpha_0)
  :=\limsup_{k\to\infty}
  \left(\frac{\norm{g_k}}{\norm{g_0}}\right)^{1/k}.
\end{equation}

Let $E_A$ be the projection-valued spectral measure of $A$.  For $g\ne0$, define the normalized scalar spectral measure
\begin{equation}\label{eq:hilb-spectral-measure}
  \nu_g(B)=\frac{\norm{E_A(B)g}^2}{\norm{g}^2},
  \qquad B\subset\spec(A)\ \text{Borel}.
\end{equation}
This is a Borel probability measure on $\spec(A)$.  For every bounded real-valued Borel function $\varphi$,
\begin{equation}\label{eq:hilb-functional-calculus}
  \frac{\ip{\varphi(A)g}{g}}{\norm{g}^2}
  =\int_{\spec(A)}\varphi(\lambda)\,\nu_g(d\lambda).
\end{equation}
We call
\begin{equation}\label{eq:hilb-active-support}
  K(g)=\supp\nu_g
\end{equation}
the active spectral support of $g$.  If $K\subset(0,\infty)$ is nonempty and compact, set
\begin{equation}\label{eq:hilb-cK}
  c(K)=
  \begin{cases}
    \dfrac{\max K-\min K}{\max K+\min K},& |K|\ge2,\\[1.2ex]
    0,& |K|=1.
  \end{cases}
\end{equation}
In particular,
\begin{equation}\label{eq:hilb-cA}
  \cA=c(\spec(A))=\frac{M-m}{M+m}.
\end{equation}

If $\spec(A)=\{\lambda\}$, then $A=\lambda I$ and every quadratic BB1 or BB2 trajectory terminates in at most two updates, independently of the positive first step.  In this case $\cA=0$.  Statements involving $\log c(K)$ below assume that $K$ has at least two points.

For $g_0\ne0$, the matched initialization is
\begin{equation}\label{eq:hilb-matched}
  \alpha_0^{\rm mat}=\frac{\norm{g_0}^2}{\ip{Ag_0}{g_0}}.
\end{equation}

\begin{theorem}\label{thm:hilbert-main}
Let the sequence be generated by \eqref{eq:hilb-BB1} with arbitrary $\alpha_0>0$.  Then, for every $g_0\ne0$,
\begin{equation}\label{eq:hilb-active-bound}
  \rho_A(g_0,\alpha_0)\le c(K(g_0))\le\cA.
\end{equation}
Under the matched initialization \eqref{eq:hilb-matched}, every $\gamma\in(\cA,1)$ admits a constant $C=C(A,\gamma)<\infty$ such that
\begin{equation}\label{eq:hilb-uniform-envelope}
  \norm{g_k}\le C\gamma^k\norm{g_0}
  \qquad\text{for every }g_0\ne0\text{ and every }k\ge0.
\end{equation}
No $\gamma\in(0,\cA)$ has this property.  Hence
\begin{equation}\label{eq:hilb-uniform-threshold}
  \inf\{\gamma\in(0,1):\gamma\text{ satisfies \eqref{eq:hilb-uniform-envelope}}\}=\cA.
\end{equation}
If $a<b$ are eigenvalues of $A$, $g_0\ne0$ is supported on their eigenspaces, and
\begin{equation}\label{eq:hilb-endpoint-balance}
  \norm{E_A(\{a\})g_0}^2
  =\norm{E_A(\{b\})g_0}^2,
\end{equation}
then the matched sequence satisfies
\begin{equation}\label{eq:hilb-exact-orbit}
  \norm{g_k}
  =\left(\frac{b-a}{b+a}\right)^k\norm{g_0}
  \qquad(k\ge0).
\end{equation}
In particular, if $m<M$ and both $m$ and $M$ belong to the point spectrum, then
\[
  \sup_{g_0\ne0,\,\alpha_0>0}\rho_A(g_0,\alpha_0)=\cA,
\]
and the supremum is attained by the matched balanced endpoint orbit.
\end{theorem}

A purely continuous-spectrum trajectory with root factor $\cA$ is given in \cref{rem:hilbert-abstract-maximizer}.

\subsection{Proof architecture and principal consequences}\label{subsec:hilb-consequences}

The Hilbert-space proof follows the finite-dimensional architecture in scalar-spectral-measure form: ergodic decomposition and a countable-base argument identify a fixed persistent support, shrinking endpoint bands replace endpoint coordinates in the coboundary certificate, and compactification together with semi-uniformity yields the trajectory and uniform-envelope bounds.  The complete proof is given in \cref{app:hilbert-proofs}.

OpenAI language models assisted Shutai Yang in preparing an initial draft of this extension from the finite-dimensional proof architecture that he had developed.

For Hilbert-space weighted rules, let $\Psi:\spec(A)\to(0,\infty)$ be Borel and impose the uniform two-sided bound
\begin{equation}\label{eq:hilb-weight-bounds}
  0<\psi_-\le\Psi(\lambda)\le\psi_+<\infty
  \qquad(\lambda\in\spec(A)).
\end{equation}
Then the conjugating operator $\Psi(A)^{1/2}$ is boundedly invertible.
Consider
\begin{equation}\label{eq:hilb-weighted-rule}
  \alpha_k^\Psi
  =\frac{\ip{\Psi(A)g_{k-1}}{g_{k-1}}}
         {\ip{\Psi(A)Ag_{k-1}}{g_{k-1}}},
  \qquad k\ge1.
\end{equation}
Under \eqref{eq:hilb-weight-bounds}, the operator $\Psi(A)A$ is uniformly positive, so the quotient is well defined whenever $g_{k-1}\ne0$.  Its matched initialization is
\[
  \alpha_0^{\Psi,\rm mat}
  =\frac{\ip{\Psi(A)g_0}{g_0}}
         {\ip{\Psi(A)Ag_0}{g_0}}.
\]

\begin{corollary}\label{cor:hilb-weighted}
Assume \eqref{eq:hilb-weight-bounds} and use \eqref{eq:hilb-weighted-rule} for $k\ge1$.  For every $g_0\ne0$ and $\alpha_0>0$, the weighted trajectory has root factor at most $c(K(g_0))\le\cA$.

Under the matched weighted initialization, its optimal uniform-envelope threshold is $\cA$.  More precisely, for $\gamma\in(\cA,1)$, if $C(A,\gamma)$ is a BB1 envelope constant in \eqref{eq:hilb-uniform-envelope}, then
\begin{equation}\label{eq:hilb-weighted-prefactor-transfer}
  \norm{g_k}
  \le \sqrt{\frac{\psi_+}{\psi_-}}\,
      C(A,\gamma)\gamma^k\norm{g_0}.
\end{equation}
If $a<b$ are eigenvalues, $g_0\ne0$ is supported on their eigenspaces, and
\begin{equation}\label{eq:hilb-weighted-balance}
  \Psi(a)\norm{E_A(\{a\})g_0}^2
  =\Psi(b)\norm{E_A(\{b\})g_0}^2,
\end{equation}
then the matched weighted sequence satisfies the exact identity \eqref{eq:hilb-exact-orbit}.  For $\Psi(\lambda)=\lambda$ this is BB2, with matched initialization $\alpha_0=\ip{Ag_0}{g_0}/\norm{Ag_0}^2$ and endpoint balance $a\norm{E_A(\{a\})g_0}^2=b\norm{E_A(\{b\})g_0}^2$.
\end{corollary}

\begin{corollary}\label{cor:hilb-uniform-family}
Let $\{\mathcal H_h,A_h\}_{h\in\mathcal I}$ be any family of real Hilbert spaces and bounded self-adjoint operators satisfying
\begin{equation}\label{eq:uniform-family-bounds}
  mI_h\le A_h\le MI_h
  \qquad(h\in\mathcal I)
\end{equation}
for fixed $0<m\le M<\infty$.  Set $c_{m,M}=(M-m)/(M+m)$.  Under matched initialization, every $\gamma\in(c_{m,M},1)$ admits a constant $C=C(m,M,\gamma)<\infty$, independent of $h$, the initial vector, and $k$, such that every BB1 gradient sequence satisfies
\begin{equation}\label{eq:uniform-family-envelope}
  \norm{g_k^{(h)}}\le C\gamma^k\norm{g_0^{(h)}}
  \qquad(h\in\mathcal I,\ g_0^{(h)}\ne0,\ k\ge0).
\end{equation}
The same conclusion holds for BB2 under its matched initialization, with a possibly larger constant depending only on $m$, $M$, and $\gamma$.
\end{corollary}

For every rule covered above, the argument in \cref{cor:error-objective} applies to its actual Hilbert-space gradient sequence.  If $g_0\ne0$, $e_k=x_k-x_*$, and $\Delta_k=Q(x_k)-Q(x_*)$, then
\[
\begin{gathered}
  g_k=Ae_k,\qquad
  \Delta_k=\frac12\ip{A^{-1}g_k}{g_k},\\
  m\norm{e_k}\le\norm{g_k}\le M\norm{e_k},\qquad
  \frac1{2M}\norm{g_k}^2\le\Delta_k\le\frac1{2m}\norm{g_k}^2.
\end{gathered}
\]
Hence the error norm and the actual gradient norm have the same root factor, and the objective-gap root factor is the square of this common value.  Proofs of \cref{thm:hilbert-main,cor:hilb-weighted,cor:hilb-uniform-family} are included in \cref{app:hilbert-proofs}.

\section{Nonlinear localization at a nondegenerate minimizer}\label{sec:nonlinear}

We now localize the sharp quadratic threshold near a nondegenerate minimizer.

Earlier frozen-quadratic analyses use stronger smoothness.  Liu and Dai sketched the pure-BB1 comparison under $C^3$ regularity in the proof of their hybrid-method theorem \cite[proof of Theorem~4.2]{LiuDai2001}.  Dai and Liao recorded the resulting $R$-linear conclusion for BB1 under the same regularity, conditional on convergence to a stationary point with positive-definite Hessian \cite[Sec.~3]{DaiLiao2002}.  Complete fixed-horizon analyses were later proved for cyclic BB1 in finite dimensions and for iteration-dependent BB1/BB2 choices in Hilbert space under a locally Lipschitz Hessian; their comparison lemmas control the nonlinear--quadratic discrepancy by $O(\norm{e_0}^2)$ until the frozen model reaches the prescribed contraction threshold \cite[Lemma~2.2 and Theorem~2.3]{DaiHagerSchittkowskiZhang2006} and \cite[Lemma~2.2 and Theorem~2.2]{AzmiKunisch2022}.

Here strict Fr\'echet differentiability of the gradient is enough.  Uniform two-point linearization and a pre-hitting lower bound for the frozen secants yield an $o(\norm{e_0})$ fixed-horizon comparison; the strict gap between the quadratic threshold and the prescribed envelope rate then permits restarts in uniformly bounded blocks.

The nonlinear frozen-model strategy and the block-restart framework were formulated by Shutai Yang; OpenAI language models assisted him in reorganizing parts of the proof and proposing candidate derivations for selected technical estimates.

\subsection{Strict linearization and the local BB iterations}\label{subsec:nonlinear-setting}

Let $\mathcal H$ be a real Hilbert space, let $U\subset\mathcal H$ be open, and let $F:U\to\R$ be continuously Fr\'echet differentiable with gradient $G:U\to\mathcal H$.  Fix $x_*\in U$ with $G(x_*)=0$.  Assume that there exists a bounded, self-adjoint, uniformly positive operator $A_*$.  Define its exact spectral endpoints by
\begin{equation}\label{eq:nonlinear-A-bounds}
\begin{aligned}
  m_*&:=\min\spec(A_*)=\norm{A_*^{-1}}^{-1}>0,\\
  M_*&:=\max\spec(A_*)=\norm{A_*}<\infty.
\end{aligned}
\end{equation}
Thus $0<m_*\le M_*$ and $m_*I\le A_*\le M_*I$.  Since $U$ is open, choose $r_0>0$ such that the closed ball
\[
  B_{r_0}(x_*):=\{x\in\mathcal H:\norm{x-x_*}\le r_0\}
\]
is contained in $U$; throughout this section, $B_r(x_*)$ denotes this closed ball.  For $0<r\le r_0$, set
\begin{equation}\label{eq:strict-modulus}
  \omega(r)
  :=\sup_{\substack{x,y\in B_r(x_*)\\x\ne y}}
  \frac{\norm{G(x)-G(y)-A_*(x-y)}}{\norm{x-y}}.
\end{equation}
The function $\omega$ is nondecreasing.  Assume that
\begin{equation}\label{eq:strict-linearization}
  \omega(r)\longrightarrow0\qquad(r\downarrow0).
\end{equation}
This two-point condition is the strict Fr\'echet differentiability of $G$ at $x_*$; in particular, it implies ordinary Fr\'echet differentiability there with derivative $A_*$ and makes $\omega(r)$ finite for all sufficiently small $r$.  After decreasing $r_0$ if necessary, assume throughout that $\omega(r_0)<m_*$.  If $F\in C^2$ near $x_*$ and $F''$ is continuous at $x_*$ in operator norm, then \eqref{eq:strict-linearization} holds with $A_*=F''(x_*)$; see \cref{cor:C2-nonlinear} below.

For $s,y\in\mathcal H$ with $s\ne0$ and $\ip{s}{y}>0$, define
\begin{equation}\label{eq:two-BB-functionals}
  \Phi_1(s,y)=\frac{\norm{s}^2}{\ip{s}{y}},
  \qquad
  \Phi_2(s,y)=\frac{\ip{s}{y}}{\norm{y}^2}.
\end{equation}
The subscripts refer to BB1 and BB2.  Fix one rule $\ell\in\{1,2\}$ throughout the iteration and choose distinct initial points $x_{-1},x_0$.  At iteration $k$, first test whether $G(x_k)=0$.  If so, terminate and keep all subsequent iterates equal to $x_k$.  Otherwise set
\begin{equation}\label{eq:nonlinear-BB}
  s_{k-1}=x_k-x_{k-1},
  \qquad
  y_{k-1}=G(x_k)-G(x_{k-1}),
  \qquad
  \alpha_k=\Phi_\ell(s_{k-1},y_{k-1}),
\end{equation}
and update
\begin{equation}\label{eq:nonlinear-BB-update}
  x_{k+1}=x_k-\alpha_kG(x_k),
  \qquad k\ge0.
\end{equation}
The local results below guarantee that every quotient actually invoked is well defined and positive.  For the frozen quadratic model $G(x)=A_*(x-x_*)$, write $g_k=A_*(x_k-x_*)$.  The initial secant $s_{-1}=x_0-x_{-1}$ is arbitrary, subject only to $s_{-1}\ne0$.  If $g_0=0$, the stopping test terminates the method before any quotient is evaluated.  Otherwise, since $y_{-1}=A_*s_{-1}$, the initial secant determines the first quadratic BB step directly:
\[
  \alpha_0=
  \begin{cases}
    \displaystyle\frac{\norm{s_{-1}}^2}{\ip{s_{-1}}{A_*s_{-1}}},
      & \ell=1,\\[1.2ex]
    \displaystyle\frac{\ip{s_{-1}}{A_*s_{-1}}}{\norm{A_*s_{-1}}^2},
      & \ell=2.
  \end{cases}
\]
For every later secant generated by an executed update, namely for $k\ge1$ before termination,
\[
  s_{k-1}=x_k-x_{k-1}=-\alpha_{k-1}g_{k-1},
  \qquad
  y_{k-1}=A_*s_{k-1}.
\]
The zero-homogeneity of the two secant quotients then gives
\[
\begin{aligned}
  \Phi_1(s_{k-1},y_{k-1})
  &=\frac{\norm{g_{k-1}}^2}{\ip{A_*g_{k-1}}{g_{k-1}}},\\
  \Phi_2(s_{k-1},y_{k-1})
  &=\frac{\ip{A_*g_{k-1}}{g_{k-1}}}{\norm{A_*g_{k-1}}^2}.
\end{aligned}
\]
Thus the arbitrary initial secant enters through the first secant quotient, whereas every subsequently generated secant yields exactly the delayed BB1 or BB2 Rayleigh rule used in the quadratic sections.

Write
\begin{equation}\label{eq:local-sharp-constant}
  \kappa_*:=\kappa(A_*)=\frac{M_*}{m_*},
  \qquad
  c_*:=\frac{M_*-m_*}{M_*+m_*}
  =\frac{\kappa_*-1}{\kappa_*+1}.
\end{equation}
For a nonnegative scalar sequence $\{z_k\}$ and $r>0$, define
\begin{equation}\label{eq:generic-root-factor}
  \operatorname{rf}(z;r)
  :=\limsup_{k\to\infty}\left(\frac{z_k}{r}\right)^{1/k}.
\end{equation}
For any distinct initial pair, set
\[
  R_0=\max\{\norm{x_{-1}-x_*},\norm{x_0-x_*}\}>0
\]
and use
\begin{equation}\label{eq:nonlinear-root-factor-definitions}
  \rho_e=\operatorname{rf}(\norm{x_k-x_*};R_0),
  \qquad
  \rho_g=\operatorname{rf}(\norm{G(x_k)};R_0).
\end{equation}
Once the objective gap is nonnegative on a tail, define
\begin{equation}\label{eq:nonlinear-objective-root-factor}
  \rho_F
  :=\limsup_{k\to\infty}
  \left(\frac{F(x_k)-F(x_*)}{R_0^2}\right)^{1/k}
\end{equation}
using any such tail.  A finitely terminating sequence is assigned root factor zero.

\begin{theorem}\label{thm:nonlinear-localization}
Assume \eqref{eq:nonlinear-A-bounds}--\eqref{eq:strict-linearization}, fix $\ell\in\{1,2\}$, and use the corresponding pure BB$\ell$ rule.  For every
\begin{equation}\label{eq:gamma-above-cstar}
  \gamma\in(c_*,1)
\end{equation}
there exist radii $0<r_\gamma\le\overline r_\gamma$, and constants
$C_\gamma,\widetilde C_\gamma,C_\gamma'<\infty$, depending only on the fixed local data $A_*$, $r_0$, and $\omega|_{(0,r_0]}$, on $\gamma$, and on the chosen rule $\ell$, with the following property.  Whenever
\begin{equation}\label{eq:local-initial-pair}
  x_{-1},x_0\in B_{r_\gamma}(x_*),
  \qquad x_{-1}\ne x_0,
\end{equation}
the iteration is well defined: at each index, either the stopping test terminates the method or the quotient in \eqref{eq:nonlinear-BB} is defined and positive.  The resulting iterates satisfy
\begin{equation}\label{eq:nonlinear-explicit-ball}
  x_k\in B_{\overline r_\gamma}(x_*)
  \qquad(k\ge-1)
\end{equation}
and
\begin{align}
  \norm{x_k-x_*}
  &\le C_\gamma\gamma^k
  \max\{\norm{x_{-1}-x_*},\norm{x_0-x_*}\},
  \label{eq:nonlinear-error-envelope}\\
  \norm{G(x_k)}
  &\le \widetilde C_\gamma\gamma^k
  \max\{\norm{x_{-1}-x_*},\norm{x_0-x_*}\}
  \label{eq:nonlinear-gradient-envelope}
\end{align}
for every $k\ge0$.  The larger radius $\overline r_\gamma$ may be chosen so that $x_*$ is the unique minimizer of $F$ in $B_{\overline r_\gamma}(x_*)$, and
\begin{equation}\label{eq:nonlinear-value-envelope}
  0\le F(x_k)-F(x_*)
  \le C_\gamma'\gamma^{2k}
  \max\{\norm{x_{-1}-x_*},\norm{x_0-x_*}\}^2.
\end{equation}

For every well-defined trajectory generated by this fixed rule and converging to $x_*$,
\begin{equation}\label{eq:nonlinear-root-bounds}
  \rho_e\le c_*,\qquad
  \rho_g\le c_*,\qquad
  \rho_F\le c_*^2.
\end{equation}
For the fixed rule, the threshold is sharp over the class of objectives satisfying the hypotheses above with the prescribed exact derivative endpoints $m_*$ and $M_*$: no smaller nonnegative number can replace $c_*$.
\end{theorem}

If $\spec(A_*)=\{\lambda\}$, then $A_*=\lambda I$ and $c_*=0$.  The theorem then says that every prescribed $\gamma\in(0,1)$ is a local envelope rate and that every such convergent trajectory has root factor zero, that is, it converges faster than every fixed geometric rate.

\subsection{Local secant geometry}\label{subsec:local-secant}

For $s\ne0$, define the two frozen-quadratic secant steps
\begin{equation}\label{eq:frozen-secant-steps}
  \widehat\alpha_1(s)
  =\frac{\norm{s}^2}{\ip{s}{A_*s}},
  \qquad
  \widehat\alpha_2(s)
  =\frac{\ip{s}{A_*s}}{\norm{A_*s}^2}.
\end{equation}

\begin{lemma}\label{lem:local-secant-geometry}
Let $0<r\le r_0$ be such that $\omega(r)<m_*$, and let distinct
$x,y\in B_r(x_*)$ be given.  With
\[
  s=x-y,\qquad v=G(x)-G(y),\qquad \varepsilon_r=\omega(r),
\]
one has
\begin{align}
  (m_*-\varepsilon_r)\norm{s}^2
  &\le \ip{s}{v}
  \le (M_*+\varepsilon_r)\norm{s}^2,
  \label{eq:local-monotonicity}\\
  (m_*-\varepsilon_r)\norm{s}
  &\le \norm{v}
  \le (M_*+\varepsilon_r)\norm{s}.
  \label{eq:local-Lipschitz-bounds}
\end{align}
Consequently both nonlinear BB steps are positive and satisfy
\begin{align}
  \frac1{M_*+\varepsilon_r}
  &\le \Phi_1(s,v)
  \le \frac1{m_*-\varepsilon_r},
  \label{eq:local-BB1-compact-bounds}\\
  \frac{m_*-\varepsilon_r}{(M_*+\varepsilon_r)^2}
  &\le \Phi_2(s,v)
  \le \frac1{m_*-\varepsilon_r}.
  \label{eq:local-BB2-compact-bounds}
\end{align}
After reducing the admissible radius so that $\omega(r)\le m_*/2$, there exists a constant $C_{\rm sec}$ depending only on $m_*$ and $M_*$ such that
\begin{equation}\label{eq:secant-step-perturbation}
  \left|\Phi_\ell(s,v)-\widehat\alpha_\ell(s)\right|
  \le C_{\rm sec}\omega(r),
  \qquad \ell\in\{1,2\},
\end{equation}
where
\begin{equation}\label{eq:explicit-Csec}
  C_{\rm sec}
  =\max\left\{\frac{2}{m_*^2},
                \frac{14M_*^2}{m_*^4}\right\}.
\end{equation}

Furthermore, for every $x\in B_r(x_*)$, with $e=x-x_*$,
\begin{align}
  (m_*-\varepsilon_r)\norm{e}
  &\le\norm{G(x)}
  \le(M_*+\varepsilon_r)\norm{e},
  \label{eq:local-gradient-equivalence}\\
  \frac{m_*-\varepsilon_r}{2}\norm{e}^2
  &\le F(x)-F(x_*)
  \le\frac{M_*+\varepsilon_r}{2}\norm{e}^2.
  \label{eq:local-value-equivalence}
\end{align}
\end{lemma}

\begin{proof}
By the definition of $\omega(r)$,
\begin{equation}\label{eq:secant-remainder-decomposition}
  v=A_*s+d,
  \qquad
  \norm{d}\le\varepsilon_r\norm{s}.
\end{equation}
Since $m_*I\le A_*\le M_*I$, Cauchy--Schwarz gives
\[
  \ip{s}{v}
  =\ip{s}{A_*s}+\ip{s}{d}
  \in
  \bigl[(m_*-\varepsilon_r)\norm{s}^2,
        (M_*+\varepsilon_r)\norm{s}^2\bigr],
\]
which is \eqref{eq:local-monotonicity}.  Similarly, the triangle and reverse triangle inequalities give
\[
  \norm{v}
  \le\norm{A_*s}+\norm{d}
  \le(M_*+\varepsilon_r)\norm{s}
\]
and
\[
  \norm{v}
  \ge\norm{A_*s}-\norm{d}
  \ge(m_*-\varepsilon_r)\norm{s},
\]
which proves \eqref{eq:local-Lipschitz-bounds}.  The BB1 bounds in
\eqref{eq:local-BB1-compact-bounds} follow immediately.  For BB2, the lower bound follows by combining the lower bound for $\ip{s}{v}$ with the upper bound for $\norm{v}$, while Cauchy--Schwarz and the lower bound for $\norm{v}$ give
\[
  \Phi_2(s,v)
  =\frac{\ip{s}{v}}{\norm{v}^2}
  \le\frac{\norm{s}}{\norm{v}}
  \le\frac1{m_*-\varepsilon_r}.
\]

For BB1,
\begin{align*}
  \left|
  \Phi_1(s,v)-\widehat\alpha_1(s)
  \right|
  &=\norm{s}^2
    \frac{|\ip{s}{d}|}
    {\ip{s}{v}\,\ip{s}{A_*s}}\le
    \frac{\varepsilon_r}
    {m_*(m_*-\varepsilon_r)}
  \le\frac{2}{m_*^2}\varepsilon_r
\end{align*}
when $\varepsilon_r\le m_*/2$.

For BB2, set
\[
  p=\ip{s}{A_*s},\qquad
  q=\norm{A_*s}^2,
\]
and
\[
  \delta=\ip{s}{d},\qquad
  h=2\ip{A_*s}{d}+\norm{d}^2.
\]
Then
\[
  \Phi_2(s,v)=\frac{p+\delta}{q+h},
  \qquad
  \widehat\alpha_2(s)=\frac pq,
\]
and hence
\begin{equation}\label{eq:BB2-secant-difference-algebra}
  \left|
  \Phi_2(s,v)-\widehat\alpha_2(s)
  \right|
  =\frac{|q\delta-ph|}{q(q+h)}.
\end{equation}
The remainder bounds imply
\[
  |\delta|\le\varepsilon_r\norm{s}^2,
  \qquad
  |h|\le(2M_*\varepsilon_r+\varepsilon_r^2)\norm{s}^2.
\]
Moreover,
\[
  m_*\norm{s}^2\le p\le M_*\norm{s}^2,
  \qquad
  m_*^2\norm{s}^2\le q\le M_*^2\norm{s}^2,
\]
and
\[
  q+h=\norm{v}^2
  \ge(m_*-\varepsilon_r)^2\norm{s}^2.
\]
Substitution in \eqref{eq:BB2-secant-difference-algebra} gives
\begin{align*}
  \left|
  \Phi_2(s,v)-\widehat\alpha_2(s)
  \right|
  &\le
  \frac{M_*^2\varepsilon_r
        +M_*(2M_*\varepsilon_r+\varepsilon_r^2)}
       {m_*^2(m_*-\varepsilon_r)^2}.
\end{align*}
If $\varepsilon_r\le m_*/2\le M_*/2$, the numerator is at most
$(7/2)M_*^2\varepsilon_r$ and the denominator is at least $m_*^4/4$.  This proves \eqref{eq:secant-step-perturbation} with the explicit constant in \eqref{eq:explicit-Csec}.

Taking $y=x_*$ in \eqref{eq:strict-modulus}, and using $G(x_*)=0$, gives
\[
  G(x_*+e)=A_*e+r_e,
  \qquad
  \norm{r_e}\le\varepsilon_r\norm{e},
\]
which proves \eqref{eq:local-gradient-equivalence}.  Finally, the segment
$x_*+te$, $0\le t\le1$, lies in $B_r(x_*)$, and
\[
  G(x_*+te)=tA_*e+r_t,
  \qquad
  \norm{r_t}\le t\varepsilon_r\norm{e}.
\]
Therefore
\begin{align*}
  F(x_*+e)-F(x_*)
  &=\int_0^1\ip{G(x_*+te)}{e}\,dt
   =\frac12\ip{A_*e}{e}
    +\int_0^1\ip{r_t}{e}\,dt,
\end{align*}
and
\[
  \left|\int_0^1\ip{r_t}{e}\,dt\right|
  \le\int_0^1t\varepsilon_r\norm{e}^2\,dt
  =\frac{\varepsilon_r}{2}\norm{e}^2.
\]
Combining this with
$m_*\norm{e}^2\le\ip{A_*e}{e}\le M_*\norm{e}^2$
proves \eqref{eq:local-value-equivalence}.
\end{proof}

\begin{lemma}\label{lem:frozen-step-directional-Lipschitz}
There are constants $L_1,L_2<\infty$, depending only on $m_*$ and $M_*$, such that for all nonzero $u,v\in\mathcal H$,
\begin{equation}\label{eq:frozen-step-normalized-Lipschitz}
  |\widehat\alpha_\ell(u)-\widehat\alpha_\ell(v)|
  \le L_\ell
  \left\|\frac{u}{\norm{u}}-\frac{v}{\norm{v}}\right\|
  \le
  \frac{2L_\ell\norm{u-v}}{\min\{\norm{u},\norm{v}\}},
  \qquad \ell\in\{1,2\}.
\end{equation}
\end{lemma}

\begin{proof}
Both maps are zero-homogeneous, so only the directions of $u$ and $v$ matter.  On the unit sphere, write
\[
  a(w)=\ip{w}{A_*w},
  \qquad
  b(w)=\norm{A_*w}^2=\ip{w}{A_*^2w}.
\]
For unit vectors $w,z$,
\begin{align*}
  |a(w)-a(z)|
  &\le 2M_*\norm{w-z},\\
  |b(w)-b(z)|
  &\le 2M_*^2\norm{w-z}.
\end{align*}
Moreover,
\[
  m_*\le a(w)\le M_*,
  \qquad
  m_*^2\le b(w)\le M_*^2.
\]
It follows that
\begin{align*}
  \left|\frac1{a(w)}-\frac1{a(z)}\right|
  &\le\frac{2M_*}{m_*^2}\norm{w-z},\\
  \left|\frac{a(w)}{b(w)}-\frac{a(z)}{b(z)}\right|
  &\le
  \left(\frac{2M_*}{m_*^2}
        +\frac{2M_*^3}{m_*^4}\right)\norm{w-z}.
\end{align*}
Thus one may take
\[
  L_1=\frac{2M_*}{m_*^2},
  \qquad
  L_2=\frac{2M_*}{m_*^2}+\frac{2M_*^3}{m_*^4}.
\]
Applying these estimates to $w=u/\norm{u}$ and $z=v/\norm{v}$ proves the first inequality in \eqref{eq:frozen-step-normalized-Lipschitz}.  For the second, assume without loss of generality that $\norm{u}\le\norm{v}$.  Then
\begin{align*}
  \left\|\frac{u}{\norm{u}}-\frac{v}{\norm{v}}\right\|
  &\le
  \frac{\norm{u-v}}{\norm{u}}
  +\norm{v}\left|\frac1{\norm{u}}-\frac1{\norm{v}}\right|
   \le\frac{2\norm{u-v}}{\norm{u}},
\end{align*}
and the claimed bound follows by symmetry.
\end{proof}

\subsection{Quadratic envelopes with a bounded first step}\label{subsec:bounded-first-step}

The nonlinear restart below produces a frozen first step from a secant joining two nonlinear iterates.  It is therefore important that the quadratic envelope be uniform over a compact set of first steps, rather than only under matched initialization.

\begin{proposition}\label{prop:bounded-first-step-envelope}
Let $A=A^*$ be bounded and uniformly positive, and set
\[
  m=\min\spec(A),\qquad M=\max\spec(A).
\]
Fix $\ell\in\{1,2\}$ and let $I\Subset(0,\infty)$ be compact.  Consider the quadratic iteration
\begin{equation}\label{eq:quadratic-bounded-first-step}
  \widehat e_{k+1}=(I-\widehat\alpha_kA)\widehat e_k,
  \qquad \widehat\alpha_0\in I,
\end{equation}
where, for $k\ge1$, $\widehat\alpha_k$ is the pure BB$\ell$ step computed from
$\widehat s_{k-1}=\widehat e_k-\widehat e_{k-1}$ and
$A\widehat s_{k-1}$.  Set
\[
  c_A=\frac{M-m}{M+m}.
\]
For every $\beta\in(c_A,1)$ there is $C=C(A,I,\beta,\ell)<\infty$ such that
\begin{equation}\label{eq:bounded-first-step-envelope}
  \norm{\widehat e_k}
  \le C\beta^k\norm{\widehat e_0}
  \qquad(k\ge0)
\end{equation}
for every $\widehat e_0\in\mathcal H$ and every $\widehat\alpha_0\in I$.
\end{proposition}

\begin{proof}
Use the zero extension after termination and define
\begin{equation}\label{eq:bounded-first-step-LI}
  L_I:=\sup_{\alpha\in I,\,\lambda\in\spec(A)}|1-\alpha\lambda|<\infty.
\end{equation}
The first update therefore satisfies, for both the error and the gradient,
\begin{equation}\label{eq:bounded-first-step-first-update}
  \norm{\widehat e_1}\le L_I\norm{\widehat e_0},
  \qquad
  \norm{\widehat g_1}\le L_I\norm{\widehat g_0},
  \qquad
  \widehat g_k=A\widehat e_k.
\end{equation}

If $\spec(A)=\{\lambda\}$, then $A=\lambda I$.  Either $\widehat e_1=0$, or $\widehat s_0=-\widehat\alpha_0\lambda\widehat e_0\ne0$ and both BB secant quotients equal $1/\lambda$, so $\widehat e_2=0$.  Thus \eqref{eq:bounded-first-step-envelope} holds with
\[
  C=\max\left\{1,\frac{L_I}{\beta}\right\}.
\]
Assume henceforth that $m<M$, so $c_A>0$.

Consider first BB1.  By \eqref{eq:bounded-first-step-first-update} and the zero extension, termination at either of the first two updates is covered by the prefactor $\max\{1,L_I/\beta\}$.  We may therefore assume $\widehat g_1,\widehat g_2\ne0$; then the first generated secant is nonzero and
\[
  \chi_1=(\nu_{\widehat g_1},\nu_{\widehat g_0})
\]
is a regular initial state of the two-step Hilbert-space system in \cref{subsec:hilb-dynamics}, because its one-step multiplier is $\norm{\widehat g_2}/\norm{\widehat g_1}>0$.  Choose $\eta>0$ so that
\begin{equation}\label{eq:eta-between-cA-beta}
  c_Ae^\eta<\beta.
\end{equation}
For every $k\ge2$ before termination, the states
$\chi_1,\ldots,\chi_{k-1}$ form a nonterminal segment, and the semi-uniform estimate \eqref{eq:hilb-semiuniform} gives
\begin{align}
  \norm{\widehat g_k}
  &\le e^{B_{\spec(A),\eta}}
       (c_Ae^\eta)^{k-1}\norm{\widehat g_1}
  \le L_Ie^{B_{\spec(A),\eta}}
       (c_Ae^\eta)^{k-1}\norm{\widehat g_0}.
       \label{eq:bounded-first-step-gradient-envelope}
\end{align}
The same bound holds at and after termination under the zero extension.  Since
\[
  m\norm{\widehat e_k}
  \le\norm{\widehat g_k}
  \le M\norm{\widehat e_k},
\]
we obtain for $k\ge2$
\begin{align*}
  \norm{\widehat e_k}
  &\le
  \frac{ML_Ie^{B_{\spec(A),\eta}}}{m}
  (c_Ae^\eta)^{k-1}\norm{\widehat e_0}
  \le
  \frac{ML_Ie^{B_{\spec(A),\eta}}}{m\beta}
  \beta^k\norm{\widehat e_0},
\end{align*}
Thus for BB1 one may take
\begin{equation}\label{eq:explicit-bounded-first-step-C}
  C_1(A,I,\beta)
  =\max\left\{
      1,\frac{L_I}{\beta},
      \frac{ML_Ie^{B_{\spec(A),\eta}}}{m\beta}
    \right\}.
\end{equation}

For BB2, define the transformed errors
\begin{equation}\label{eq:BB2-bounded-first-step-transform}
  \widetilde e_k=A^{1/2}\widehat e_k,
  \qquad
  \widetilde s_{k-1}=A^{1/2}\widehat s_{k-1}.
\end{equation}
The update is still
\[
  \widetilde e_{k+1}=(I-\widehat\alpha_kA)\widetilde e_k.
\]
For every generated nonzero secant,
\begin{align*}
  \frac{\norm{\widetilde s_{k-1}}^2}
       {\ip{\widetilde s_{k-1}}{A\widetilde s_{k-1}}}
  &=\frac{\ip{\widehat s_{k-1}}{A\widehat s_{k-1}}}
          {\ip{\widehat s_{k-1}}{A^2\widehat s_{k-1}}}
  =\frac{\ip{\widehat s_{k-1}}{A\widehat s_{k-1}}}
          {\norm{A\widehat s_{k-1}}^2},
\end{align*}
which is exactly the BB2 step of the original sequence.  Hence
$\{\widetilde e_k\}$ is a BB1 sequence for the same operator $A$ and the same arbitrary first step $\widehat\alpha_0\in I$.  Applying the BB1 estimate and using
\[
  \sqrt m\norm{\widehat e_k}
  \le\norm{\widetilde e_k}
  \le\sqrt M\norm{\widehat e_k}
\]
gives
\[
  \norm{\widehat e_k}
  \le\sqrt{\frac{M}{m}}\,C_1(A,I,\beta)
       \beta^k\norm{\widehat e_0}.
\]
This proves the proposition uniformly over $\widehat\alpha_0\in I$ for both fixed rules.
\end{proof}

\subsection{Finite-horizon comparison with the frozen quadratic model}\label{subsec:finite-horizon}

Fix a rule $\ell\in\{1,2\}$.  For a local pair $e_{-1}=x_{-1}-x_*$ and $e_0=x_0-x_*$, let $e_j=x_j-x_*$ denote the nonlinear sequence.  Its frozen comparison sequence is defined by
\begin{equation}\label{eq:frozen-comparison-initial}
  \widehat e_0=e_0,
  \qquad
  \widehat\alpha_0=\widehat\alpha_\ell(e_0-e_{-1}),
  \qquad
  \widehat e_1=(I-\widehat\alpha_0A_*)\widehat e_0,
\end{equation}
and, for $j\ge1$, by the pure quadratic BB$\ell$ recurrence
\begin{equation}\label{eq:frozen-comparison-recurrence}
  \widehat s_{j-1}=\widehat e_j-\widehat e_{j-1},
  \qquad
  \widehat\alpha_j=\widehat\alpha_\ell(\widehat s_{j-1}),
  \qquad
  \widehat e_{j+1}=(I-\widehat\alpha_jA_*)\widehat e_j.
\end{equation}
As in \cref{prop:bounded-first-step-envelope}, if the frozen sequence terminates, it is kept at zero and no subsequent secant quotient is evaluated.

\begin{lemma}\label{lem:finite-horizon-tracking}
Fix an integer $N\ge1$ and $\tau\in(0,1)$.  There exist $r_{N,\tau}>0$, $D_{N,\tau}<\infty$, and a nondecreasing function
$\varepsilon_{N,\tau}:(0,r_{N,\tau}]\to[0,\infty)$ such that $D_{N,\tau}$ depends only on $N,\tau,m_*,M_*$ and the fixed rule $\ell$, while $r_{N,\tau}$ and $\varepsilon_{N,\tau}$ also depend on the local modulus $\omega$, and
\begin{equation}\label{eq:tracking-modulus-limit}
  \varepsilon_{N,\tau}(r)\longrightarrow0
  \qquad(r\downarrow0)
\end{equation}
and
\begin{equation}\label{eq:tracking-modulus-half-tau}
  \varepsilon_{N,\tau}(r)\le\frac{\tau}{2}
  \qquad(0<r\le r_{N,\tau}).
\end{equation}
Suppose
\[
  R:=\max\{\norm{e_{-1}},\norm{e_0}\}
  \le r\le r_{N,\tau},
  \qquad e_{-1}\ne e_0,
  \qquad e_0\ne0,
\]
and let $1\le q\le N$.  If
\begin{equation}\label{eq:pre-hitting-lower-bound}
  \norm{\widehat e_j}\ge\tau\norm{e_0}
  \qquad(0\le j<q),
\end{equation}
then the nonlinear iteration is well defined through the update producing $e_q$, and
\begin{align}
  \max_{0\le j\le q}\norm{e_j-\widehat e_j}
  &\le\varepsilon_{N,\tau}(r)\norm{e_0},
  \label{eq:finite-horizon-tracking}\\
  \max_{0\le j\le q}
  \bigl(\norm{e_j}+\norm{\widehat e_j}\bigr)
  &\le D_{N,\tau}\norm{e_0}.
  \label{eq:finite-horizon-growth}
\end{align}
\end{lemma}

\begin{proof}
For either frozen rule and every nonzero secant,
\begin{equation}\label{eq:frozen-step-compact-interval}
  a_-:=\frac1{M_*}
  \le\widehat\alpha_j\le
  \frac1{m_*}=:a_+.
\end{equation}
For BB1 this is the usual Rayleigh-quotient bound.  For BB2 it follows from
$m_*A_*\le A_*^2\le M_*A_*$, which gives
\[
  \frac1{M_*}
  \le
  \frac{\ip{s}{A_*s}}{\ip{s}{A_*^2s}}
  \le\frac1{m_*}.
\]
Set
\begin{equation}\label{eq:frozen-one-step-growth-constant}
  L_*:=\max_{\substack{\alpha\in[a_-,a_+]\\
                         \lambda\in[m_*,M_*]}}
       |1-\alpha\lambda|,
  \qquad
  D_0:=\max_{0\le j\le N}L_*^j.
\end{equation}
The first frozen step and all subsequent frozen steps lie in the same interval, so the zero extension after termination satisfies
\begin{equation}\label{eq:frozen-finite-growth}
  \max_{0\le j\le N}\norm{\widehat e_j}
  \le D_0\norm{e_0}.
\end{equation}

Define
\begin{equation}\label{eq:bootstrap-thresholds}
  d_*:=a_-m_*\tau=\frac{m_*\tau}{M_*},
  \qquad
  \delta_*:=\min\left\{1,\frac\tau2,\frac{d_*}{4}\right\},
  \qquad
  C_0:=\max\{1,D_0+\delta_*\}.
\end{equation}
For every $j<q$, the condition \eqref{eq:pre-hitting-lower-bound} and \eqref{eq:frozen-step-compact-interval} imply
\begin{align}
  \norm{\widehat s_j}
  &=\widehat\alpha_j\norm{A_*\widehat e_j}
  \ge a_-m_*\norm{\widehat e_j}
  \ge d_*\norm{e_0}.
  \label{eq:frozen-secant-lower-bound}
\end{align}

Put
\begin{equation}\label{eq:tracking-error-definition}
  E_j=e_j-\widehat e_j,
  \qquad
  \delta_j=
  \max_{0\le i\le j}\frac{\norm{E_i}}{\norm{e_0}}.
\end{equation}
Temporarily suppose that $r$ is small enough that
\begin{equation}\label{eq:temporary-local-smallness}
  C_0r\le r_0,
  \qquad
  \omega(C_0r)\le\frac{m_*}{2}.
\end{equation}
The final choice of $r_{N,\tau}$ below will enforce these conditions.  Write
\[
  \Omega_r:=\omega(C_0r).
\]
By \eqref{eq:local-BB1-compact-bounds}--\eqref{eq:local-BB2-compact-bounds}, every nonlinear step evaluated while its two secant endpoints remain in $B_{C_0r}(x_*)$ satisfies the common upper bound
\begin{equation}\label{eq:actual-step-upper}
  0<\alpha_j\le\frac1{m_*-\Omega_r}\le\frac2{m_*}.
\end{equation}

Because $e_0\ne0$ and $\Omega_r<m_*$, \eqref{eq:local-gradient-equivalence} gives
$\norm{G(x_*+e_0)}\ge(m_*-\Omega_r)\norm{e_0}>0$; hence the initial stopping test does not terminate the method.  Since the two initial points lie in $B_r(x_*)\subset B_{C_0r}(x_*)$, \eqref{eq:secant-step-perturbation} gives
\begin{equation}\label{eq:first-step-comparison}
  |\alpha_0-\widehat\alpha_0|
  \le C_1\Omega_r,
  \qquad C_1:=C_{\rm sec}.
\end{equation}
Moreover,
\begin{align}
  E_1
  &=e_1-\widehat e_1
  =-(\alpha_0-\widehat\alpha_0)A_*e_0
    -\alpha_0\bigl(G(x_*+e_0)-A_*e_0\bigr).
  \label{eq:first-error-identity}
\end{align}
The strict linearization estimate and \eqref{eq:actual-step-upper} therefore yield
\begin{align*}
  \norm{E_1}
  &\le
  M_*C_1\Omega_r\norm{e_0}
  +\frac2{m_*}\Omega_r\norm{e_0}.
\end{align*}
Thus
\begin{equation}\label{eq:first-iterate-comparison}
  \delta_1\le B_1\Omega_r,
  \qquad
  B_1:=M_*C_1+\frac2{m_*}.
\end{equation}

Suppose that the nonlinear and frozen sequences have been constructed through an index
$j\in\{1,\ldots,q-1\}$ and that
\begin{equation}\label{eq:bootstrap-hypothesis}
  \delta_j\le\delta_*.
\end{equation}
For every $0\le i\le j$, \eqref{eq:frozen-finite-growth} and \eqref{eq:bootstrap-hypothesis} imply
\begin{align}
  \norm{\widehat e_i}
  &\le D_0\norm{e_0},\notag\\
  \norm{e_i}
  &\le\norm{\widehat e_i}+\norm{E_i}
   \le(D_0+\delta_*)\norm{e_0}
   \le C_0\norm{e_0}
   \le C_0r.
  \label{eq:all-points-in-bootstrap-ball}
\end{align}
Thus all local secant and gradient estimates may be applied on $B_{C_0r}(x_*)$.

The actual and frozen generated secants satisfy
\begin{align}
  s_{j-1}-\widehat s_{j-1}
  &=(e_j-e_{j-1})-(\widehat e_j-\widehat e_{j-1})
  =E_j-E_{j-1},
\end{align}
and hence
\begin{equation}\label{eq:generated-secant-difference}
  \norm{s_{j-1}-\widehat s_{j-1}}
  \le2\delta_j\norm{e_0}.
\end{equation}
By \eqref{eq:frozen-secant-lower-bound} with index $j-1$ and
$\delta_j\le\delta_*\le d_*/4$,
\begin{align}
  \norm{\widehat s_{j-1}}
  &\ge d_*\norm{e_0},\notag\\
  \norm{s_{j-1}}
  &\ge\norm{\widehat s_{j-1}}
       -\norm{s_{j-1}-\widehat s_{j-1}}
  \ge(d_*-2\delta_j)\norm{e_0}
  \ge\frac{d_*}{2}\norm{e_0}.
  \label{eq:nonlinear-secant-lower-bound}
\end{align}
In particular, the actual secant is nonzero.  Also, since $j<q$,
\begin{equation}\label{eq:nonlinear-prehit-lower-bound}
  \norm{e_j}
  \ge\norm{\widehat e_j}-\norm{E_j}
  \ge(\tau-\delta_*)\norm{e_0}
  \ge\frac\tau2\norm{e_0}>0.
\end{equation}
The local gradient equivalence therefore gives $G(x_*+e_j)\ne0$, so the stopping test at index $j$ does not terminate the method, and \cref{lem:local-secant-geometry} makes the next nonlinear quotient well defined and positive.

We now estimate the difference between the nonlinear and frozen steps.  Insert the frozen quotient evaluated at the actual secant:
\begin{align}
  |\alpha_j-\widehat\alpha_j|
  &\le
  \left|\Phi_\ell(s_{j-1},y_{j-1})
       -\widehat\alpha_\ell(s_{j-1})\right|+
  \left|\widehat\alpha_\ell(s_{j-1})
       -\widehat\alpha_\ell(\widehat s_{j-1})\right|.
  \label{eq:step-error-splitting}
\end{align}
The first term is at most $C_{\rm sec}\Omega_r$.  For the second, use
\eqref{eq:frozen-step-normalized-Lipschitz},
\eqref{eq:generated-secant-difference}, and
\eqref{eq:nonlinear-secant-lower-bound}:
\begin{align*}
  \left|\widehat\alpha_\ell(s_{j-1})
       -\widehat\alpha_\ell(\widehat s_{j-1})\right|
  &\le
  \frac{2L_\ell\norm{s_{j-1}-\widehat s_{j-1}}}
       {\min\{\norm{s_{j-1}},\norm{\widehat s_{j-1}}\}}
  \le
  \frac{2L_\ell\,2\delta_j\norm{e_0}}
       {(d_*/2)\norm{e_0}}
  =\frac{8L_\ell}{d_*}\delta_j.
\end{align*}
Consequently
\begin{equation}\label{eq:successive-step-comparison}
  |\alpha_j-\widehat\alpha_j|
  \le A_0\delta_j+B_0\Omega_r,
  \qquad
  A_0:=\frac{8L_\ell}{d_*},
  \quad B_0:=C_{\rm sec}.
\end{equation}
Here, $A_0\delta_j$ measures the error in the secant direction, whereas $B_0\Omega_r$ measures the nonlinear deviation of the gradient secant from $A_*s_{j-1}$.

Let
\[
  R_j:=G(x_*+e_j)-A_*e_j,
  \qquad
  \norm{R_j}\le\Omega_r\norm{e_j}.
\]
Subtracting the two updates gives the exact identity
\begin{align}
  E_{j+1}
  &=(I-\widehat\alpha_jA_*)E_j
    -(\alpha_j-\widehat\alpha_j)A_*e_j
    -\alpha_jR_j.
  \label{eq:tracking-recurrence}
\end{align}
By \eqref{eq:frozen-one-step-growth-constant},
\eqref{eq:all-points-in-bootstrap-ball},
\eqref{eq:actual-step-upper}, and
\eqref{eq:successive-step-comparison},
\begin{align*}
  \frac{\norm{E_{j+1}}}{\norm{e_0}}
  &\le
  L_*\delta_j
  +(A_0\delta_j+B_0\Omega_r)M_*C_0
  +\frac2{m_*}\Omega_r C_0\\
  &=\bigl(L_*+A_0M_*C_0\bigr)\delta_j
    +C_0\left(M_*B_0+\frac2{m_*}\right)\Omega_r.
\end{align*}
Since
\[
  \delta_{j+1}
  =\max\left\{\delta_j,
       \frac{\norm{E_{j+1}}}{\norm{e_0}}\right\},
\]
one may take
\begin{equation}\label{eq:tracking-AB-explicit}
  A:=\max\left\{1,L_*+A_0M_*C_0\right\},
  \qquad
  B:=C_0\left(M_*B_0+\frac2{m_*}\right),
\end{equation}
and obtain
\begin{equation}\label{eq:tracking-scalar-recurrence}
  \delta_{j+1}\le A\delta_j+B\Omega_r.
\end{equation}
Here $A$ and $B$ depend only on $N,\tau,m_*,M_*$ and the fixed rule.

Set
\begin{equation}\label{eq:tracking-K-constant}
  \overline B:=\max\{B,B_1\},
  \qquad
  K_{N,\tau}:=\overline B
  \sum_{i=0}^{N-1}A^i.
\end{equation}
Because $\omega(t)\to0$, we can choose $r_{N,\tau}>0$ small enough that
\begin{equation}\label{eq:bootstrap-radius-choice}
  C_0r_{N,\tau}\le r_0,
  \qquad
  \omega(C_0r_{N,\tau})\le\frac{m_*}{2},
  \qquad
  K_{N,\tau}\omega(C_0r_{N,\tau})\le\delta_*.
\end{equation}
We claim that, simultaneously with the construction of the nonlinear iterates through index $q$,
\begin{equation}\label{eq:tracking-induction-formula}
  \delta_j
  \le\overline B\Omega_r
      \sum_{i=0}^{j-1}A^i
  \le K_{N,\tau}\Omega_r
  \le\delta_*,
  \qquad 1\le j\le q.
\end{equation}
For $j=1$, this follows from \eqref{eq:first-iterate-comparison}.  Suppose it holds through some $j<q$.  The last inequality in \eqref{eq:tracking-induction-formula} verifies the hypothesis needed to construct the next step.  Applying \eqref{eq:tracking-scalar-recurrence} and using $B\le\overline B$ gives
\begin{align*}
  \delta_{j+1}
  &\le
  A\overline B\Omega_r\sum_{i=0}^{j-1}A^i
  +\overline B\Omega_r
  =\overline B\Omega_r\sum_{i=0}^{j}A^i,
\end{align*}
This proves the induction and validates the construction through the update producing $e_q$.

For every $j<q$, \eqref{eq:nonlinear-prehit-lower-bound} shows that the nonlinear iterate is nonstationary.  The initial secant is nonzero by hypothesis, and every later secant has the form
\[
  s_{j-1}=-\alpha_{j-1}G(x_*+e_{j-1})\ne0.
\]
Hence every quotient required before the update producing $e_q$ is well defined and positive.

Finally, set
\begin{equation}\label{eq:tracking-epsilon-and-D}
  \varepsilon_{N,\tau}(r)
  :=K_{N,\tau}\omega(C_0r),
  \qquad
  D_{N,\tau}:=2D_0+\delta_*.
\end{equation}
The function $\varepsilon_{N,\tau}$ is nondecreasing and tends to zero.  By \eqref{eq:bootstrap-radius-choice}, it is at most $\delta_*\le\tau/2$. \eqref{eq:tracking-induction-formula} gives \eqref{eq:finite-horizon-tracking}, while
\begin{align*}
  \norm{e_j}+\norm{\widehat e_j}
  &\le\norm{E_j}+2\norm{\widehat e_j}
  \le(\delta_*+2D_0)\norm{e_0}
  =D_{N,\tau}\norm{e_0}
\end{align*}
proves \eqref{eq:finite-horizon-growth}.
\end{proof}

\subsection{Proof of the sharp local threshold}\label{subsec:nonlinear-proof}

\begin{proof}[Proof of \cref{thm:nonlinear-localization}]
Write $e_k=x_k-x_*$.  If $e_0=0$, the stopping test terminates the method at $x_*$ before any secant quotient is evaluated.  Assume henceforth that $e_0\ne0$.

Fix $\gamma\in(c_*,1)$ and choose
\begin{equation}\label{eq:beta-gamma-choice}
  c_*<\beta<\gamma.
\end{equation}
For every nonzero secant $s$, the frozen first step
$\widehat\alpha_\ell(s)$ belongs to a compact interval $I_\ell=[1/M_*,1/m_*]\Subset(0,\infty)$.  By
\cref{prop:bounded-first-step-envelope}, there exists $C_\beta<\infty$ such that every frozen comparison sequence satisfies
\begin{equation}\label{eq:frozen-beta-envelope}
  \norm{\widehat e_j}
  \le C_\beta\beta^j\norm{\widehat e_0}
  \qquad(j\ge0),
\end{equation}
uniformly over its secant-generated first step.

Choose $N\in\mathbb N$ so large that
\begin{equation}\label{eq:N-for-sharp-localization}
  C_\beta\left(\frac\beta\gamma\right)^N\le\frac14,
\end{equation}
and set
\begin{equation}\label{eq:tau-choice}
  \tau=\frac14\gamma^N.
\end{equation}
Then
\[
  \norm{\widehat e_N}
  \le C_\beta\beta^N\norm{\widehat e_0}
  \le\tau\norm{\widehat e_0}.
\]
Since $\norm{\widehat e_0}>\tau\norm{\widehat e_0}$, every frozen orbit has a first index $q\in\{1,\ldots,N\}$ for which
\begin{equation}\label{eq:first-hitting-definition}
  \norm{\widehat e_q}\le\tau\norm{\widehat e_0},
  \qquad
  \norm{\widehat e_j}>\tau\norm{\widehat e_0}
  \quad(0\le j<q).
\end{equation}
This remains true if the frozen orbit terminates before $N$, because it is then extended by zero.

Apply \cref{lem:finite-horizon-tracking} with this $N$ and $\tau$.  Shrink a radius $r\le r_{N,\tau}$, if necessary, so that the local gradient and value estimates hold on $B_r(x_*)$ and $\omega(r)<m_*$.  Consider any distinct pair with
\[
  \max\{\norm{e_{-1}},\norm{e_0}\}\le r.
\]
At the first frozen hitting time, the tracking estimate and
$\varepsilon_{N,\tau}(r)\le\tau/2$ give
\begin{align}
  \norm{e_q}
  \le\norm{\widehat e_q}
      +\norm{e_q-\widehat e_q}
  \le\left(\tau+\frac\tau2\right)\norm{e_0}
   =\frac38\gamma^N\norm{e_0}
  \le\gamma^q\norm{e_0}.
  \label{eq:one-block-sharp-contraction}
\end{align}
The same lemma also gives, for every $0\le j\le q$,
\begin{equation}\label{eq:one-block-growth}
  \norm{e_j}\le D_{N,\tau}\norm{e_0}.
\end{equation}
Thus a block of at most $N$ steps contracts its current error by a factor no greater than $\gamma^q$, while all intermediate errors are controlled by a fixed multiple of the error at the beginning of that block.

Let
\begin{equation}\label{eq:block-growth-D}
  D:=\max\{1,D_{N,\tau}\},
\end{equation}
and define
\begin{equation}\label{eq:nonlinear-radius-choice}
  r_\gamma=\frac rD,
  \qquad
  \overline r_\gamma=r=Dr_\gamma.
\end{equation}
The smaller initial radius $r_\gamma=r/D$ guarantees, through the within-block growth bound, that every block remains inside the larger ball.

Set $k_0=0$.  We construct consecutive block endpoints $k_i$ such that
\begin{equation}\label{eq:restart-induction-hypothesis}
  \norm{e_{k_i}}
  \le\gamma^{k_i}\norm{e_0}.
\end{equation}
For $i=0$ this is an equality, and the original pair
$(e_{-1},e_0)$ lies in the smaller ball and hence in the larger one.
Suppose a nonterminal block endpoint $k_i$ has been constructed.  For $i\ge1$, the point immediately preceding the restart belongs to the previous block.  If that block began at $k_{i-1}$, then \eqref{eq:one-block-growth} and \eqref{eq:restart-induction-hypothesis} give
\begin{align}
  \norm{e_{k_i-1}}
  &\le D\norm{e_{k_{i-1}}}
   \le D\gamma^{k_{i-1}}\norm{e_0}
   \le D\norm{e_0}
   \le r,
  \label{eq:restart-predecessor-radius}\\
  \norm{e_{k_i}}
  &\le\gamma^{k_i}\norm{e_0}
   \le\norm{e_0}
   \le r_\gamma
   \le r.
  \label{eq:restart-current-radius}
\end{align}
Thus the restarted pair lies in the larger ball.

If $e_{k_i}=0$, the method has terminated.  Otherwise the restarted pair is distinct.  Relabel the restarted pair as a new pair $(e_{-1}^{(i)},e_0^{(i)})=(e_{k_i-1},e_{k_i})$.  Its frozen first step is
\[
  \widehat\alpha_\ell(e_{k_i}-e_{k_i-1})\in I_\ell.
\]
Applying the same argument to the restarted pair produces
$q_i\in\{1,\ldots,N\}$.  Set
$k_{i+1}=k_i+q_i$.  Then
\begin{align}
  \norm{e_{k_{i+1}}}
  &\le\gamma^{q_i}\norm{e_{k_i}}
  \le\gamma^{k_{i+1}}\norm{e_0}.
  \label{eq:selected-subsequence-rate}
\end{align}
This closes the induction.

For an index $k_i\le k<k_{i+1}$, the within-block growth bound gives
\begin{align}
  \norm{e_k}
  \le D\norm{e_{k_i}}
  \le D\gamma^{k_i}\norm{e_0}
  =D\gamma^{-(k-k_i)}\gamma^k\norm{e_0}
  \le D\gamma^{-N}\gamma^k\norm{e_0}.
  \label{eq:within-block-rate}
\end{align}
In particular,
\[
  \norm{e_k}\le D\norm{e_0}\le Dr_\gamma=r
\]
for every $k\ge0$, while the two prescribed initial points are already in
$B_{r_\gamma}(x_*)$.  Therefore all iterates lie in
$B_{\overline r_\gamma}(x_*)$ and the block construction also proves that every quotient invoked by the nonlinear method is well defined and positive.

Taking
\begin{equation}\label{eq:explicit-Cgamma}
  C_\gamma=D\gamma^{-N}
\end{equation}
and using $\norm{e_0}\le R_0$ proves
\eqref{eq:nonlinear-error-envelope}.  On $B_r(x_*)$,
\eqref{eq:local-gradient-equivalence} gives
\[
  \norm{G(x_k)}
  \le(M_*+\omega(r))C_\gamma\gamma^kR_0,
\]
so one may take
\begin{equation}\label{eq:explicit-gradient-prefactor}
  \widetilde C_\gamma=(M_*+\omega(r))C_\gamma.
\end{equation}
Likewise, \eqref{eq:local-value-equivalence} gives
\[
  0\le F(x_k)-F(x_*)
  \le\frac{M_*+\omega(r)}2C_\gamma^2\gamma^{2k}R_0^2,
\]
so one may take
\begin{equation}\label{eq:explicit-value-prefactor}
  C_\gamma'=\frac{M_*+\omega(r)}2C_\gamma^2.
\end{equation}
Because $m_*-\omega(r)>0$, the lower value bound is strict away from
$x_*$, hence $x_*$ is the unique minimizer of $F$ in the larger ball.

Let a well-defined trajectory generated by the fixed rule converge to $x_*$.  Fix $\gamma\in(c_*,1)$.  If the trajectory does not terminate, then for all sufficiently large $K$ the consecutive pair
$(x_{K-1},x_K)$ lies in $B_{r_\gamma}(x_*)$.  Relabeling this consecutive pair as $(x_{-1},x_0)$ preserves its secant quotient, so the relabeled iteration is exactly the original tail.  The estimates above therefore give
\[
  \rho_e\le\gamma,
  \qquad
  \rho_g\le\gamma,
  \qquad
  \rho_F\le\gamma^2.
\]
Letting $\gamma\downarrow c_*$ proves \eqref{eq:nonlinear-root-bounds}.

For sharpness, assume that $m_*<M_*$, take $x_*=0$, and set
\[
  A_*=\diag(m_*,M_*)
  \qquad\text{on }\mathbb R^2.
\]
For the quadratic objective $Q(e)=\tfrac12\ip{A_*e}{e}$, the balanced BB1 endpoint orbit in \cref{prop:endpoint-orbit}, with $a=m_*$ and $b=M_*$, has exact gradient factor $c_*$.  Its BB2 counterpart is given by \cref{cor:weighted-main} with $\Psi(t)=t$ and the weighted balance \eqref{eq:weighted-balance}.  Thus no number smaller than $c_*$ is a valid uniform threshold.

For a genuinely nonquadratic example with the same sharp factor, fix $\ell\in\{1,2\}$ and write $e=(u,v)$.  In error coordinates, the endpoint orbit just cited alternates between
\[
  L_\pm^{(\ell)}=\operatorname{span}(1,\pm r_\ell),
  \qquad
  r_1=\frac{m_*}{M_*},
  \qquad
  r_2=\left(\frac{m_*}{M_*}\right)^{3/2}.
\]
Define the linear functionals by
\[
  \varphi_+^{(\ell)}(u,v)=v-r_\ell u,
  \qquad
  \varphi_-^{(\ell)}(u,v)=v+r_\ell u.
\]
Then $\ker\varphi_\pm^{(\ell)}=L_\pm^{(\ell)}$, and their product is the elementary quadratic polynomial
\[
  \varphi_+^{(\ell)}(u,v)\varphi_-^{(\ell)}(u,v)
  =v^2-r_\ell^2u^2.
\]
For any $\eta>0$, set
\[
  P_\ell(u,v)=\bigl(v^2-r_\ell^2u^2\bigr)^2,
  \qquad
  F_{\eta,\ell}(u,v)
  =\frac12\bigl(m_*u^2+M_*v^2\bigr)+\eta P_\ell(u,v).
\]
The polynomial $P_\ell$ is nonzero and homogeneous of degree four, so this objective is genuinely nonquadratic and belongs to $C^\infty$.  Moreover,
\[
  \nabla P_\ell(u,v)
  =4\bigl(v^2-r_\ell^2u^2\bigr)(-r_\ell^2u,v),
\]
so $P_\ell=0$ and $\nabla P_\ell=0$ on
$L_+^{(\ell)}\cup L_-^{(\ell)}$.  Since also $D^2P_\ell(0)=0$, one has $F_{\eta,\ell}''(0)=A_*$.  Thus $F_{\eta,\ell}$ satisfies the hypotheses of the theorem.  In addition,
\[
  F_{\eta,\ell}(e)
  \ge\frac{m_*}{2}\norm{e}^2,
\]
so the origin is its unique global minimizer.

Given $t>0$, choose
\[
  e_{-1}=t(1,r_\ell),
  \qquad
  e_0=t c_*(1,-r_\ell).
\]
These are two consecutive points of the cited quadratic endpoint orbit.  At every point of that orbit, $\nabla F_{\eta,\ell}(e)=A_*e$.  Hence the nonlinear gradient secants, BB$\ell$ quotients, and updates agree exactly with their quadratic counterparts, and induction gives
\[
  e_k=t c_*^{k+1}\bigl(1,(-1)^{k+1}r_\ell\bigr)
  \qquad(k\ge0).
\]
  Letting $t\downarrow0$ gives such a trajectory in every neighborhood of the minimizer.  Consequently this genuinely nonquadratic trajectory has
$\rho_e=\rho_g=c_*$ and $\rho_F=c_*^2$.
\end{proof}

\begin{corollary}\label{cor:C2-nonlinear}
Let $F\in C^2$ in a neighborhood of $x_*$ and assume
\[
\begin{aligned}
  G(x_*)&=0,\ A_*=F''(x_*)=A_*^*,\
  m_*:=\min\spec(A_*)>0,\
  M_*:=\max\spec(A_*).
\end{aligned}
\]
Then, for both pure BB1 and pure BB2, the local-envelope and convergent-trajectory conclusions of \cref{thm:nonlinear-localization} hold for $F$.  If $m_*<M_*$, the value $c_*$ is sharp already within the $C^\infty$ nonquadratic subclass constructed in the proof of \cref{thm:nonlinear-localization}.
\end{corollary}

\begin{proof}
Here $C^2$ means twice continuously Fr\'echet differentiable, so $F''$ is operator-norm continuous.  For $x,y$ near $x_*$, the line segment from $y$ to $x$ remains in the $C^2$ neighborhood and
\[
  G(x)-G(y)-A_*(x-y)
  =\int_0^1
  \bigl(F''(y+t(x-y))-A_*\bigr)(x-y)\,dt.
\]
Taking the supremum over a sufficiently small ball and using operator-norm continuity of $F''$ at $x_*$ gives \eqref{eq:strict-linearization}.
\end{proof}

\section{Concluding discussion}\label{sec:conclusion}

For each fixed pure BB rule on a finite-dimensional positive-definite quadratic, every positive first step has gradient root factor at most the active-endpoint constant, whose global value is $c_H=(\kappa(H)-1)/(\kappa(H)+1)$; when the global endpoints are distinct, matched initialization and endpoint balance attain $c_H$, which is also the matched uniform-envelope threshold.  This sharpens the earlier componentwise and Property~B estimates \cite{LiSun2021,LiHuang2025} by identifying the exact norm-level threshold.

In Hilbert space, the same $c_A=(M-m)/(M+m)$ is the sharp matched uniform-envelope threshold for each fixed rule, and a trajectory with an arbitrary positive first step is controlled by its own active scalar spectral support.  When $m<M$, the balanced endpoint state is realized by a vector if and only if both endpoint eigenspaces are nontrivial, while special operators with purely continuous spectrum can also have trajectories whose root factor equals $c_A$; see \cref{rem:hilbert-abstract-maximizer}.

For nonlinear objectives, every $\gamma\in(c_*,1)$ is a local envelope rate under strict Fr\'echet differentiability of the gradient, lowering the locally Lipschitz-Hessian regularity used in earlier complete fixed-horizon analyses \cite{DaiHagerSchittkowskiZhang2006,AzmiKunisch2022}.  The trajectory bounds allow an arbitrary initial secant: every well-defined trajectory converging to $x_*$ has error, gradient, and objective-gap root factors at most $c_*$, $c_*$, and $c_*^2$.  Over the class of objectives with prescribed distinct derivative endpoints $m_*<M_*$, matched endpoint trajectories for quadratic and $C^\infty$ genuinely nonquadratic examples in $\R^2$ attain these factors.

Alternative delayed spectral steps can first be compared through their normalized spectral dynamics.  Endpoint cycles and a coboundary certificate then diagnose the maximal invariant average.  The conjugacy results show that fixed positive spectral reweighting alone cannot lower the sharp threshold; any improvement must change the extremal dynamics.

\section*{Acknowledgements}
Shutai Yang thanks Shixiang Chen for supervising his undergraduate thesis at the University of Science and Technology of China, and Xiaowei Xu for supervising his project under the National Undergraduate Training Program for Innovation and Entrepreneurship at the same university.  The present article grew out of those two undergraduate projects and substantially extends their mathematical content.

\section*{Statements and Declarations}

\subsection*{Funding}
This work of Shutai Yang was supported by the National Undergraduate Training Program for Innovation and Entrepreneurship, administered by the University of Science and Technology of China (Project No.~202510358091).

\subsection*{Competing interests}
The authors have no relevant financial or non-financial interests to disclose.

\subsection*{Declaration on the use of generative AI}
During the preparation of this manuscript, Shutai Yang used OpenAI language models for language and \LaTeX{} editing, the presentation of mathematical notation, and the section-specific drafting and proof-development assistance disclosed above.  He independently checked and revised all model-assisted material, made all scholarly judgments and final decisions, and takes full responsibility for the manuscript.

\subsection*{Data and code availability}
No datasets were generated or analyzed, and no computational code was used in this entirely theoretical study.

\appendix
\section{Proofs for the Hilbert-space extension}\label{app:hilbert-proofs}

We now give the detailed proof of \cref{thm:hilbert-main} and its weighted and family-uniform consequences.

\subsection{The two-step dynamics of scalar spectral measures}\label{subsec:hilb-dynamics}

Fix a nonempty compact set $K\subset(0,\infty)$ and let $\cP(K)$ be the space of Borel probability measures on $K$, equipped with the weak topology.  The space
\begin{equation}\label{eq:hilb-XK}
  X_K=\cP(K)\times\cP(K)
\end{equation}
is compact and metrizable.  For $\zeta\in\cP(K)$, set
\begin{equation}\label{eq:hilb-u}
  u(\zeta)=\int_K\lambda\,\zeta(d\lambda).
\end{equation}
For $\chi=(\nu,\zeta)\in X_K$, define
\begin{equation}\label{eq:hilb-R-r}
  R_K(\nu,\zeta)
  =\int_K\left(1-\frac{\lambda}{u(\zeta)}\right)^2\nu(d\lambda),
  \qquad r_K(\nu,\zeta)=\sqrt{R_K(\nu,\zeta)}.
\end{equation}
On the regular set
\begin{equation}\label{eq:hilb-regular-set}
  X_K^{\rm reg}=\{\chi\in X_K:R_K(\chi)>0\},
\end{equation}
let
\begin{equation}\label{eq:hilb-T}
  T_K(\nu,\zeta)=(\nu^+,\nu),
  \qquad
  \nu^+(d\lambda)
  =\frac{(1-\lambda/u(\zeta))^2}{R_K(\nu,\zeta)}\,\nu(d\lambda).
\end{equation}
\begin{lemma}\label{lem:hilb-dynamics-continuity}
The maps $R_K$ and $r_K$ are continuous on $X_K$, and $T_K$ is continuous on $X_K^{\rm reg}$.
\end{lemma}

\begin{proof}
Let $(\nu_n,\zeta_n)\Rightarrow(\nu,\zeta)$ in $X_K$.  Since $\lambda\mapsto\lambda$ is continuous on the compact set $K$,
$u(\zeta_n)\to u(\zeta)$.  Moreover, $u(\zeta_n),u(\zeta)\ge\min K>0$, so the functions
\[
  q_n(\lambda)=\left(1-\frac{\lambda}{u(\zeta_n)}\right)^2
\]
converge uniformly on $K$ to $q(\lambda)=(1-\lambda/u(\zeta))^2$.  Therefore
\[
  \left|\int_K q_n\,d\nu_n-\int_K q\,d\nu\right|
  \le \norm{q_n-q}_\infty
      +\left|\int_K q\,d\nu_n-\int_K q\,d\nu\right|
  \longrightarrow0,
\]
which proves continuity of $R_K$ and hence of $r_K$.  If $R_K(\nu,\zeta)>0$, then $R_K(\nu_n,\zeta_n)>0$ for all sufficiently large $n$.  The same uniform-convergence argument applied to $\varphi q_n$ shows that, for every $\varphi\in C(K)$,
\[
  \int_K\varphi\,d\nu_n^+
  =\frac{\int_K\varphi q_n\,d\nu_n}{R_K(\nu_n,\zeta_n)}
  \longrightarrow
  \frac{\int_K\varphi q\,d\nu}{R_K(\nu,\zeta)}
  =\int_K\varphi\,d\nu^+.
\]
Thus $\nu_n^+\Rightarrow\nu^+$, proving continuity of $T_K$ on the regular set.
\end{proof}

\begin{lemma}\label{lem:hilb-recurrence}
Let $K\subset(0,\infty)$ be compact and contain $K(g_0)$, and view every scalar spectral measure supported in $K$ as an element of $\cP(K)$.  For every $k\ge1$ such that $g_k\ne0$ and $g_{k+1}\ne0$,
\begin{equation}\label{eq:hilb-measure-cocycle}
  (\nu_{k+1},\nu_k)=T_K(\nu_k,\nu_{k-1}),
  \qquad
  \frac{\norm{g_{k+1}}}{\norm{g_k}}
  =r_K(\nu_k,\nu_{k-1}),
  \qquad \nu_j=\nu_{g_j}.
\end{equation}
Under matched initialization, if the first update is nonterminal, the first normalized state is $\chi_0=(\nu_0,\nu_0)$ and $\chi_1=T_K\chi_0$.  For arbitrary $\alpha_0>0$, if the first update is nonterminal, the first state governed by $T_K$ is $\chi_1=(\nu_1,\nu_0)$; no preceding spectral measure is introduced.
\end{lemma}

\begin{proof}
For $k\ge1$, \eqref{eq:hilb-functional-calculus} gives $\alpha_k=1/u(\nu_{k-1})$.  Since
$g_{k+1}=(I-\alpha_kA)g_k$, the functional calculus gives, for every Borel set $B\subset K$,
\[
  \norm{E_A(B)g_{k+1}}^2
  =\norm{g_k}^2
   \int_B\left(1-\frac{\lambda}{u(\nu_{k-1})}\right)^2\nu_k(d\lambda).
\]
Taking $B=K$ gives the norm multiplier, and normalization gives the updated probability measure.  Under \eqref{eq:hilb-matched}, the same calculation for the first update uses $\alpha_0=1/u(\nu_0)$ and yields $\chi_1=T_K(\nu_0,\nu_0)$ whenever $g_1\ne0$.
\end{proof}

A multiplier cannot create spectral mass: \eqref{eq:hilb-T} gives $\nu^+\ll\nu$.  To encode finite termination, set
\[
  \Sigma_K=X_K\setminus X_K^{\rm reg}
\]
and adjoin an isolated point $\theta$.  On
\begin{equation}\label{eq:hilb-compactification}
  \widehat X_K=X_K\sqcup\{\theta\}
\end{equation}
define
\begin{equation}\label{eq:hilb-hat-map-potential}
  \widehat T_K\chi=
  \begin{cases}
    T_K\chi,&\chi\in X_K^{\rm reg},\\
    \theta,&\chi\in\Sigma_K\cup\{\theta\},
  \end{cases}
  \qquad
  \widehat f_K(\chi)=
  \begin{cases}
    \log r_K(\chi),&\chi\in X_K^{\rm reg},\\
    -\infty,&\chi\in\Sigma_K\cup\{\theta\}.
  \end{cases}
\end{equation}
Then $\widehat X_K$ is compact, $\widehat T_K$ is Borel and continuous outside $\Sigma_K$, and $\widehat f_K$ is upper semicontinuous and bounded above.

\subsection{Ergodic support reduction and endpoint bands}\label{subsec:hilb-bands}

For a $\widehat T_K$-invariant probability measure $\Pi$, the relevant case is
\begin{equation}\label{eq:hilb-finite-integral}
  \int\widehat f_K\,d\Pi> -\infty.
\end{equation}
As in the finite-dimensional argument, we first reduce an ergodic invariant measure to a single persistent support.  The finite collection of coordinate faces is replaced by the support of the barycenter of the first-coordinate marginal.

\begin{lemma}\label{lem:hilb-persistent-support}
Let $\Pi$ be an ergodic $\widehat T_K$-invariant Borel probability measure satisfying \eqref{eq:hilb-finite-integral}.  Then there exist a compact set $S\subset K$ with at least two points and a $\Pi$-full, forward-invariant Borel set $Y_S^\infty$ such that every
$\chi=(\nu,\zeta)\in Y_S^\infty$ satisfies
\begin{equation}\label{eq:hilb-equal-support}
  \supp\nu=\supp\zeta=S,
\end{equation}
and the same equality holds at every forward iterate of $\chi$.
\end{lemma}

\begin{proof}
Since
$\widehat T_K^{-1}(\{\theta\})=\Sigma_K\cup\{\theta\}$,
invariance gives $\Pi(\Sigma_K)=0$.  The finite integral in \eqref{eq:hilb-finite-integral} excludes positive mass at $\theta$, and hence
\begin{equation}\label{eq:hilb-boundary-null}
  \Pi(\Sigma_K\cup\{\theta\})=0.
\end{equation}

Let $\mathcal U$ be a countable base for the relative topology of $K$.  For $U\in\mathcal U$, define
\[
  b_U(\nu,\zeta)=\ind_{\{\nu(U)>0\}},
  \qquad b_U(\theta)=0.
\]
The map $b_U$ is Borel.  Since the first coordinate $\nu^+$ of $T_K(\nu,\zeta)$ satisfies $\nu^+\ll\nu$, one has
\[
  b_U\circ\widehat T_K\le b_U.
\]
Invariance makes the integrals of the two sides equal, so
$b_U\circ\widehat T_K=b_U$ almost surely.  Ergodicity therefore implies that $b_U$ is almost surely equal to either zero or one.

Define the barycenter of the first-coordinate marginal by
\begin{equation}\label{eq:hilb-ergodic-barycenter}
  \overline\nu(B)=\int_{X_K}\nu(B)\,d\Pi(\nu,\zeta),
  \qquad B\subset K\text{ Borel},
\end{equation}
and set $S=\supp\overline\nu$.  For every $U\in\mathcal U$,
\[
  \overline\nu(U)>0
  \quad\Longleftrightarrow\quad
  \Pi\{\nu(U)>0\}>0
  \quad\Longleftrightarrow\quad
  \nu(U)>0\ \text{almost surely}.
\]
After intersecting over the countable base, the support characterization by positive-mass neighborhoods gives
\begin{equation}\label{eq:hilb-first-support-fixed}
  \supp\nu=S
  \qquad\text{almost surely}.
\end{equation}

Extend the coordinate projections $\pi_1,\pi_2$ to $\theta$ by assigning them the same arbitrary element of $\cP(K)$.  On $X_K^{\rm reg}$,
$\pi_2\circ\widehat T_K=\pi_1$; therefore \eqref{eq:hilb-boundary-null} and invariance give
\begin{equation}\label{eq:hilb-marginals-equal}
  (\pi_2)_*\Pi=(\pi_1)_*\Pi.
\end{equation}
Thus, for every $U\in\mathcal U$, the events $\{\zeta(U)>0\}$ and $\{\nu(U)>0\}$ have the same probability.  Combining this with the preceding zero-one conclusion and again intersecting over $\mathcal U$ yields
$\supp\zeta=S$ almost surely.

The set
\[
  Y_S=\{(\nu,\zeta)\in X_K^{\rm reg}:
        \supp\nu=\supp\zeta=S\}
\]
is Borel; this follows from the same countable-base characterization of support.  Since $\Pi(Y_S)=1$, the set
\begin{equation}\label{eq:hilb-persistent-support-core}
  Y_S^\infty=\bigcap_{n\ge0}\widehat T_K^{-n}Y_S
\end{equation}
is forward invariant and has full measure.  Finally, if $S=\{\lambda\}$, then $\nu=\zeta=\delta_\lambda$ on $Y_S$, which gives $R_K=0$, contrary to $Y_S\subset X_K^{\rm reg}$.  Hence $S$ has at least two points.
\end{proof}

\begin{lemma}\label{lem:hilb-restricted-weak}
Let $S$ be a closed subset of the compact metric space $K$.  Suppose that $\nu_n,\nu\in\cP(K)$ are supported on $S$ and $\nu_n\Rightarrow\nu$ in $\cP(K)$.  Then $\nu_n\Rightarrow\nu$ in $\cP(S)$.  Consequently,
\begin{equation}\label{eq:relative-continuity-set}
  \nu_n(B)\longrightarrow\nu(B)
\end{equation}
for every Borel set $B\subset S$ whose relative boundary in $S$ has zero $\nu$-measure.
\end{lemma}

\begin{proof}
Every function in $C(S)$ extends continuously to $K$.  Since all the measures are supported on $S$, weak convergence on $K$ therefore implies weak convergence on $S$.  The second assertion is the continuity-set consequence of weak convergence on $S$.
\end{proof}

Let $S\subset(0,\infty)$ be compact with endpoints $a<b$, and put
\begin{equation}\label{eq:hilb-pq}
  p=\frac{b}{a+b},\qquad q=\frac{a}{a+b}.
\end{equation}
For $0<\delta<(b-a)/2$, define
\begin{equation}\label{eq:hilb-endpoint-bands}
  L_\delta(S)=S\cap[a,a+\delta],
  \qquad
  U_\delta(S)=S\cap[b-\delta,b].
\end{equation}
If $\supp\nu=S$, both masses are positive.  Set
\begin{equation}\label{eq:hilb-band-transfer}
  h_{\delta,S}(\nu)
  =-\frac p2\log\nu(L_\delta(S))
   -\frac q2\log\nu(U_\delta(S)),
\end{equation}
\begin{equation}\label{eq:hilb-band-Gamma}
  \Gamma_{\delta,S}(u)
  =\left[
      \sup_{\lambda\in L_\delta(S)}
      \left|1-\frac{\lambda}{u}\right|
    \right]^p
    \left[
      \sup_{\lambda\in U_\delta(S)}
      \left|1-\frac{\lambda}{u}\right|
    \right]^q,
\end{equation}
and
\begin{equation}\label{eq:hilb-band-C}
  C_{\delta,S}=\sup_{u\in[a,b]}\Gamma_{\delta,S}(u).
\end{equation}

\begin{lemma}\label{lem:hilb-band-certificate}
Let $\chi=(\nu,\zeta)\in X_K^{\rm reg}$ satisfy
$\supp\nu=\supp\zeta=S$, with endpoints $a<b$, and write
$T_K\chi=(\nu^+,\nu)$.  Then
\begin{equation}\label{eq:hilb-band-certificate}
  \log r_K(\chi)
  \le \log\Gamma_{\delta,S}(u(\zeta))
     +h_{\delta,S}(\nu^+)-h_{\delta,S}(\nu)
\end{equation}
and therefore
\begin{equation}\label{eq:hilb-band-uniform}
  \log r_K(\chi)
  \le \log C_{\delta,S}
     +h_{\delta,S}(\nu^+)-h_{\delta,S}(\nu).
\end{equation}
Moreover,
\begin{equation}\label{eq:hilb-band-limit}
  C_{\delta,S}\longrightarrow\frac{b-a}{b+a}
  \qquad(\delta\downarrow0).
\end{equation}
\end{lemma}

\begin{proof}
Write $u=u(\zeta)$ and $Q_u(\lambda)=(1-\lambda/u)^2$.  Since $\supp\zeta=S$ has distinct endpoints, $u\in(a,b)$.  The two band masses of $\nu$ are positive.  Moreover, $Q_u(a)>0$ and $Q_u(b)>0$; continuity and the support condition show that the conditional averages $\overline Q_L$ and $\overline Q_U$ of $Q_u$ over the two endpoint bands are strictly positive.

From \eqref{eq:hilb-T},
\[
  \nu^+(L_\delta(S))
  =\frac{\nu(L_\delta(S))\overline Q_L}{R_K(\chi)},
  \qquad
  \nu^+(U_\delta(S))
  =\frac{\nu(U_\delta(S))\overline Q_U}{R_K(\chi)}.
\]
Substitution into \eqref{eq:hilb-band-transfer} gives
\[
  h_{\delta,S}(\nu^+)-h_{\delta,S}(\nu)
  =\log r_K(\chi)-\frac p2\log\overline Q_L
                    -\frac q2\log\overline Q_U.
\]
Bounding each conditional average by the square of the corresponding supremum proves \eqref{eq:hilb-band-certificate}; taking the supremum in $u$ gives \eqref{eq:hilb-band-uniform}.

For $\lambda\in L_\delta(S)$ and $u\in[a,b]$,
\[
  \left|
    \left|1-\frac{\lambda}{u}\right|-\frac{u-a}{u}
  \right|\le\frac{\delta}{a},
\]
and the analogous estimate holds on the upper band.  Hence the two endpoint suprema converge uniformly to $(u-a)/u$ and $(b-u)/u$.  Uniform continuity of $(s,t)\mapsto s^pt^q$ on a common compact range gives uniform convergence of $\Gamma_{\delta,S}$ to
\[
  \left(\frac{u-a}{u}\right)^{b/(a+b)}
  \left(\frac{b-u}{u}\right)^{a/(a+b)}.
\]
Taking suprema and applying \cref{lem:endpoint-ineq} proves \eqref{eq:hilb-band-limit}.
\end{proof}

When $\nu$ has positive atoms at $a$ and $b$, singleton endpoint bands recover the finite-dimensional coboundary \eqref{eq:cohomology}.  The endpoint-band certificate is its continuous-support analogue.

\subsection{Invariant averages and semi-uniformity}\label{subsec:hilb-invariant}

Throughout this subsection, the ambient set $K$ is fixed.  The symbols $\nu,\zeta\in\cP(K)$ denote coordinate measures, whereas $\Pi$ denotes an invariant probability measure on the state space $\widehat X_K$; the persistent support $S\subset K$ may depend on the ergodic component of $\Pi$.

\begin{proposition}\label{prop:hilb-invariant-bound}
Let $K\subset(0,\infty)$ be compact with at least two points.  Every $\widehat T_K$-invariant probability measure $\Pi$ obeys
\begin{equation}\label{eq:hilb-invariant-bound}
  \int\widehat f_K\,d\Pi\le\log c(K).
\end{equation}
Here the integral is understood in the extended sense.
\end{proposition}

\begin{proof}
Since $\widehat f_K$ is bounded above, its integral is well defined in $[-\infty,\infty)$.  If $\int\widehat f_K\,d\Pi=-\infty$, the assertion is immediate.  Hence assume \eqref{eq:hilb-finite-integral}.

Let $\Pi=\int\eta\,d\Lambda(\eta)$ be the ergodic decomposition.  Applying the decomposition to the bounded truncations $\max\{\widehat f_K,-M\}$ and then letting $M\to\infty$ gives
\[
  \int\widehat f_K\,d\Pi
  =\int\left(\int\widehat f_K\,d\eta\right)d\Lambda(\eta).
\]
Since $\widehat f_K$ is bounded above and the left-hand side is finite, almost every ergodic component has finite logarithmic average.  It is therefore enough to prove \eqref{eq:hilb-invariant-bound} for an ergodic $\Pi$ satisfying \eqref{eq:hilb-finite-integral}.

Let $S$ and $Y_S^\infty$ be given by \cref{lem:hilb-persistent-support}, and write
\[
  a=\min S,\qquad b=\max S.
\]
Choose a point $\chi=(\nu_0,\zeta_0)\in Y_S^\infty$ that is topologically recurrent and Birkhoff generic for $\widehat f_K$.  Thus there are $n_\ell\to\infty$ such that
$\widehat T_K^{n_\ell}\chi\to\chi$, and
\[
  \frac1n\sum_{k=0}^{n-1}\widehat f_K(\widehat T_K^k\chi)
  \longrightarrow\int\widehat f_K\,d\Pi.
\]
Write $\widehat T_K^n\chi=(\nu_n,\zeta_n)$.  Every $\nu_n$ and $\zeta_n$ has support $S$.

Choose $\delta_j\downarrow0$, with $0<\delta_j<(b-a)/2$, such that
\[
  \nu_0(\{a+\delta_j\})
  =\nu_0(\{b-\delta_j\})=0.
\]
Along the recurrence subsequence, $\nu_{n_\ell}\Rightarrow\nu_0$ in $\cP(K)$.  By \cref{lem:hilb-restricted-weak}, the convergence also holds in $\cP(S)$.  The endpoint bands are continuity sets for $\nu_0$, and therefore
\[
  h_{\delta_j,S}(\nu_{n_\ell})
  \longrightarrow h_{\delta_j,S}(\nu_0).
\]
Summing \eqref{eq:hilb-band-uniform} from $0$ to $n_\ell-1$, dividing by $n_\ell$, and letting $\ell\to\infty$ gives
\[
  \int\widehat f_K\,d\Pi\le\log C_{\delta_j,S}.
\]
Finally let $j\to\infty$.  By \eqref{eq:hilb-band-limit},
\[
  \int\widehat f_K\,d\Pi
  \le\log c(S)\le\log c(K).
\]
The bound holds on almost every ergodic component and hence for the original invariant measure.
\end{proof}

For the extremal statement, write
\[
  a=\min K,\qquad b=\max K,\qquad m_K=\frac{a+b}{2},
\]
and define
\begin{equation}\label{eq:hilb-balanced-state}
  \nu_K^*=\frac12\delta_a+\frac12\delta_b,
  \qquad \chi_K^*=(\nu_K^*,\nu_K^*).
\end{equation}

\begin{proposition}\label{prop:hilb-unique-maximizer}
Let $K\subset(0,\infty)$ be compact with at least two points.  The maximal invariant average of $\widehat f_K$ equals $\log c(K)$, and the unique maximizing invariant probability measure on $\widehat X_K$ is
\begin{equation}\label{eq:hilb-unique-maximizer}
  \delta_{\chi_K^*}.
\end{equation}
\end{proposition}

\begin{proof}
The state $\chi_K^*$ is regular, fixed, and has multiplier $c(K)$, while \cref{prop:hilb-invariant-bound} gives the matching upper bound.

It remains to prove uniqueness.  We first consider an ergodic maximizing measure $\Pi$.  Let $S$ be its persistent support from \cref{lem:hilb-persistent-support}.  The ergodic part of the preceding proof gives
\begin{equation}\label{eq:hilb-max-support-chain}
  \log c(K)
  =\int\widehat f_K\,d\Pi
  \le\log c(S)\le\log c(K).
\end{equation}
Hence $c(S)=c(K)$.  Since $S\subset K$ and $(t-s)/(t+s)$ is strictly decreasing in $s$ and strictly increasing in $t$, the endpoints of $S$ are the global endpoints $a$ and $b$.

Let $\overline\nu$ be the barycenter in \eqref{eq:hilb-ergodic-barycenter}.  Choose a decreasing sequence
\begin{equation}\label{eq:hilb-max-band-deltas}
  0<\delta_j<\frac{b-a}{2},\qquad \delta_j\downarrow0,
\end{equation}
such that
\begin{equation}\label{eq:hilb-max-no-boundary-atoms}
  \overline\nu(\{a+\delta_j\})
  =\overline\nu(\{b-\delta_j\})=0.
\end{equation}
For each $j$, \eqref{eq:hilb-ergodic-barycenter} and \eqref{eq:hilb-max-no-boundary-atoms} imply
\[
  \nu(\{a+\delta_j\})=\nu(\{b-\delta_j\})=0
\]
for $\Pi$-almost every state $(\nu,\zeta)$.  Since both the indices $j\ge1$ and the forward iterates are countable, invariance yields a $\Pi$-full Borel set $Z$ such that, whenever $\chi\in Z$ and
$\widehat T_K^n\chi=(\nu_n,\zeta_n)$,
\[
  \nu_n(\{a+\delta_j\})=\nu_n(\{b-\delta_j\})=0
  \qquad(n\ge0,\ j\ge1).
\]

Set
\[
  L_j=L_{\delta_j}(S),\qquad
  U_j=U_{\delta_j}(S),\qquad
  \Gamma_j=\Gamma_{\delta_j,S},\qquad
  h_j=h_{\delta_j,S}.
\]
For $N\in\mathbb N$, define
\begin{equation}\label{eq:hilb-max-truncated-observable}
  G_{j,N}(\nu,\zeta)
  =\max\{\log\Gamma_j(u(\zeta)),-N\},
  \qquad G_{j,N}(\theta)=-N.
\end{equation}
The functions $G_{j,N}$ are bounded and Borel.  Choose
$\chi=(\nu_0,\zeta_0)\in Y_S^\infty\cap Z$ that is recurrent and is Birkhoff generic for $\widehat f_K$ and for every $G_{j,N}$.  Write
$\chi_n=(\nu_n,\zeta_n)=\widehat T_K^n\chi$ and choose recurrence times $n_\ell$ with $\chi_{n_\ell}\to\chi$.

The sets $L_j$ and $U_j$ are continuity sets for $\nu_0$ relative to $S$.  Hence \cref{lem:hilb-restricted-weak} gives
$h_j(\nu_{n_\ell})\to h_j(\nu_0)$.  Since $G_{j,N}\ge\log\Gamma_j\circ u$, summing \eqref{eq:hilb-band-certificate} along the recurrence subsequence yields
\[
  \log c(K)
  \le\int G_{j,N}\,d\Pi.
\]
For $\lambda,u\in[a,b]$,
$|1-\lambda/u|\le(b-a)/a$.  Thus all $G_{j,N}$ and $\log\Gamma_j\circ u$ have the common finite upper bound
\[
  U_0=\log\max\left\{1,\frac{b-a}{a}\right\}.
\]
Letting $N\to\infty$ by monotone convergence applied to $U_0-G_{j,N}$ gives
\begin{equation}\label{eq:hilb-max-band-average}
  \int\log\Gamma_j(u(\zeta))\,d\Pi
  \ge\log c(K).
\end{equation}
As $j\to\infty$, the endpoint bands decrease and $\Gamma_j\downarrow\Gamma_0$ pointwise, where
\[
  \Gamma_0(u)
  =\left(\frac{u-a}{u}\right)^{b/(a+b)}
   \left(\frac{b-u}{u}\right)^{a/(a+b)}.
\]
Since $\log\Gamma_j\le U_0$, monotone convergence applied to the nonnegative increasing functions $U_0-\log\Gamma_j$, with $\log 0=-\infty$, gives
\[
  \int\log\Gamma_j(u(\zeta))\,d\Pi
  \longrightarrow
  \int\log\Gamma_0(u(\zeta))\,d\Pi.
\]
Combining this limit with \eqref{eq:hilb-max-band-average} gives
\[
  \int\log\Gamma_0(u(\zeta))\,d\Pi
  \ge\log c(K).
\]
By \cref{lem:endpoint-ineq}, the integrand is at most $\log c(K)$, with equality only at $u(\zeta)=m_K$.  Hence
\begin{equation}\label{eq:hilb-max-midpoint}
  u(\zeta)=m_K
  \qquad\text{almost surely}.
\end{equation}
Equality of the coordinate marginals in \eqref{eq:hilb-marginals-equal} gives $u(\nu)=m_K$ almost surely as well.

Using \eqref{eq:hilb-max-midpoint},
\begin{equation}\label{eq:hilb-endpoint-rigidity-final}
  r_K(\nu,\zeta)^2
  =\int_K\left(1-\frac{2\lambda}{a+b}\right)^2\nu(d\lambda)
  \le c(K)^2,
\end{equation}
with equality only when $\nu$ is supported on $\{a,b\}$.  On the $\Pi$-full regular set $Y_S^\infty$, one has $\widehat f_K=\log r_K\le\log c(K)$.  Since $\int\widehat f_K\,d\Pi=\log c(K)$, it follows that $r_K=c(K)$ almost surely.  Hence equality holds almost surely in \eqref{eq:hilb-endpoint-rigidity-final}.  But $\supp\nu=S$ almost surely, so $S=\{a,b\}$.  The midpoint conditions then give
\[
  \nu=\zeta=\frac12(\delta_a+\delta_b)=\nu_K^*
  \qquad\text{almost surely}.
\]
Thus every ergodic maximizing measure equals $\delta_{\chi_K^*}$.

Finally, let $\Pi$ be an arbitrary maximizing invariant measure and write its ergodic decomposition as $\Pi=\int\eta\,d\Lambda(\eta)$.  Applying the decomposition to the bounded truncations of $\widehat f_K$ and then using monotone convergence, as in the proof of \cref{prop:hilb-invariant-bound}, gives
\[
  \int\widehat f_K\,d\Pi
  =\int\left(\int\widehat f_K\,d\eta\right)d\Lambda(\eta).
\]
By \cref{prop:hilb-invariant-bound}, the inner integral is at most $\log c(K)$ for $\Lambda$-almost every component, while the left-hand side equals $\log c(K)$.  Hence almost every component is maximizing and therefore equals $\delta_{\chi_K^*}$.  Consequently $\Pi=\delta_{\chi_K^*}$.
\end{proof}

\begin{remark}\label{rem:hilbert-abstract-maximizer}
When $K=\spec(A)$, the balanced state $\chi_K^*$ is realized by a Hilbert-space vector if and only if
\[
  E_A(\{a\})\ne0,
  \qquad
  E_A(\{b\})\ne0.
\]
Equal-norm components in the two endpoint eigenspaces then generate the exact matched geometric orbit.

Continuous-spectrum root-factor equality can also occur without endpoint eigenvectors.  Let $\mathcal H=L^2([a,b],d\lambda)$, let $(Af)(\lambda)=\lambda f(\lambda)$, and take $g_0\equiv1$.  The spectrum $[a,b]$ is purely continuous, and symmetry about $(a+b)/2$ gives, under matched BB1,
\[
  \alpha_k=\frac{2}{a+b},
  \qquad
  g_k(\lambda)=\left(1-\frac{2\lambda}{a+b}\right)^k
  \qquad(k\ge0).
\]
With $c(K)=(b-a)/(b+a)$, the change of variables
$t=(2\lambda-a-b)/(b-a)$ yields
\[
  \frac{\norm{g_k}}{\norm{g_0}}
  =\frac{c(K)^k}{\sqrt{2k+1}},
  \qquad
  \rho_A(g_0,\alpha_0^{\rm mat})=c(K).
\]
For BB2 under its matched initialization, taking $g_0(\lambda)=\lambda^{-1/2}$ makes $A^{1/2}g_0=1$, so \cref{cor:hilb-weighted} yields the same root factor $c(K)$.
\end{remark}

\begin{lemma}\label{lem:hilb-semiuniform}
Let $K\subset(0,\infty)$ be compact with at least two points.  For every $\eta>0$, there exists $B_{K,\eta}<\infty$ such that every nonterminal orbit segment, meaning
\[
  \chi,\widehat T_K\chi,\ldots,\widehat T_K^{n-1}\chi\in X_K^{\rm reg},
\]
satisfies
\begin{equation}\label{eq:hilb-semiuniform}
  \sum_{k=0}^{n-1}\widehat f_K(\widehat T_K^k\chi)
  \le n\bigl(\log c(K)+\eta\bigr)+B_{K,\eta}
  \qquad(n\ge1).
\end{equation}
\end{lemma}

\begin{proof}
Apply \cref{prop:abstract-semiuniform} with
\[
  X=\widehat X_K,\qquad T=\widehat T_K,\qquad
  D=\Sigma_K\cup\{\theta\},\qquad
  f=\widehat f_K,\qquad \beta=\log c(K).
\]
By \cref{lem:hilb-dynamics-continuity} and \eqref{eq:hilb-hat-map-potential}, $\widehat X_K$ is compact metrizable, $\widehat T_K$ is Borel with every discontinuity contained in $D$, and $\widehat f_K$ is upper semicontinuous, bounded above, and equal to $-\infty$ on $D$.  The invariant-measure hypothesis is \cref{prop:hilb-invariant-bound}.
\end{proof}

\subsection{Proof of the extension and weighted consequences}\label{subsec:hilb-proof}

\begin{proof}[Proof of \cref{thm:hilbert-main}]
Let $K_0=K(g_0)$.  By the standing zero-extension convention, singleton-support and finite-termination cases are immediate.  Thus assume that $K_0$ has at least two points and $g_1\ne0$.  The multiplier $I-\alpha_0A$ cannot create spectral mass, so $\nu_1$ is supported in $K_0$ and $\chi_1=(\nu_1,\nu_0)\in X_{K_0}$.  For every $n\ge2$ before termination, the states $\chi_1,\ldots,\chi_{n-1}$ form a nonterminal orbit segment and
\[
  \frac{\norm{g_n}}{\norm{g_0}}
  =\frac{\norm{g_1}}{\norm{g_0}}
   \prod_{k=1}^{n-1}r_{K_0}(\chi_k),
  \qquad \chi_k=\widehat T_{K_0}^{k-1}\chi_1.
\]
Thus, for every $\eta>0$, \cref{lem:hilb-semiuniform} gives
\[
  \norm{g_n}
  \le \norm{g_1}\exp(B_{K_0,\eta})
      \bigl(c(K_0)e^\eta\bigr)^{n-1}.
\]
Taking $n$th roots and then letting $\eta\downarrow0$ proves \eqref{eq:hilb-active-bound}.

If $M=m$, then the matched step terminates exactly, every $\gamma\in(0,1)$ is a uniform envelope rate, and the threshold is $0=\cA$.  Assume henceforth that $M>m$.  It remains to consider matched trajectories with nonsingleton initial support.  Their normalized orbits start from $\chi_0=(\nu_0,\nu_0)$ in the single compact state space $X_{\spec(A)}$.  Given $\gamma\in(\cA,1)$, choose $\eta=\log(\gamma/\cA)$.  For every $k\ge1$ before termination, the states $\chi_0,\ldots,\chi_{k-1}$ are nonterminal and \cref{lem:hilb-semiuniform} gives
\[
  \norm{g_k}\le \exp(B_{\spec(A),\eta})\gamma^k\norm{g_0}.
\]
Hence \eqref{eq:hilb-uniform-envelope} holds with
\[
  C=\max\{1,\exp(B_{\spec(A),\eta})\},
\]
which depends only on $A$ and $\gamma$.

To prove optimality of the uniform-envelope threshold, suppose that \eqref{eq:hilb-uniform-envelope} holds for some $C<\infty$ and some $\gamma\in(0,\cA)$.  Choose $N$ so large that
\begin{equation}\label{eq:hilb-N-choice}
  \cA^N>2C\gamma^N.
\end{equation}
For every $0<\delta<(M-m)/2$, the spectral projections
$E_A([m,m+\delta])$ and $E_A([M-\delta,M])$ are nonzero.  Indeed, if the first projection vanished, the spectral representation would give $A\ge(m+\delta)I$, contradicting $m=\min\spec(A)$; the upper endpoint is analogous.  Choose unit vectors $u_\delta$ and $v_\delta$ in these orthogonal ranges and set
\begin{equation}\label{eq:hilb-approx-vector}
  g_0^\delta=\frac{u_\delta+v_\delta}{\sqrt2}.
\end{equation}
Then $\norm{g_0^\delta}=1$, and orthogonality of the two reducing spectral subspaces gives, for every Borel set $B$,
\[
  \nu_{g_0^\delta}(B)
  =\frac12\nu_{u_\delta}(B)+\frac12\nu_{v_\delta}(B).
\]
The two measures on the right are supported in intervals shrinking to $m$ and $M$, respectively, so
\begin{equation}\label{eq:hilb-measure-convergence}
  \nu_{g_0^\delta}\Rightarrow
  \frac12\delta_m+\frac12\delta_M
  \qquad(\delta\downarrow0).
\end{equation}
Let $\chi_0^\delta=(\nu_{g_0^\delta},\nu_{g_0^\delta})$.  Then $\chi_0^\delta\Rightarrow \chi_{\spec(A)}^*$, and the limit is a regular fixed state with multiplier $\cA>0$.  By induction on $j=0,\ldots,N-1$, continuity of $T_{\spec(A)}$ and $r_{\spec(A)}$ at this fixed state implies that, for all sufficiently small $\delta$, the first $N$ states are regular, $\chi_j^\delta\Rightarrow \chi_{\spec(A)}^*$, and
\[
  r_{\spec(A)}(\chi_j^\delta)\longrightarrow\cA.
\]
Consequently,
\begin{equation}\label{eq:hilb-finite-horizon}
  \frac{\norm{g_N^\delta}}{\norm{g_0^\delta}}
  =\prod_{j=0}^{N-1}r_{\spec(A)}(\chi_j^\delta)
  \longrightarrow\cA^N.
\end{equation}
For sufficiently small $\delta$, the left-hand side exceeds $\cA^N/2>C\gamma^N$, contradicting the assumed envelope.  Hence no $\gamma\in(0,\cA)$ is uniform.

Finally, under \eqref{eq:hilb-endpoint-balance}, the matched spectral mean is $(a+b)/2$.  The two endpoint multipliers are
$(b-a)/(b+a)$ and its negative, so their squared energies remain balanced and \eqref{eq:hilb-exact-orbit} follows by induction.
\end{proof}

\begin{proof}[Proof of \cref{cor:hilb-weighted}]
Set $S=\Psi(A)^{1/2}$ and $\widetilde g_k=Sg_k$.  The bounds \eqref{eq:hilb-weight-bounds} make $S$ boundedly invertible, and $S$ commutes with $A$.  The first transformed update is
\[
  \widetilde g_1=(I-\alpha_0A)\widetilde g_0,
\]
and for every $k\ge1$ the weighted quotient is the BB1 Rayleigh quotient of $\widetilde g_{k-1}$.  Hence $\{\widetilde g_k\}$ is a BB1 gradient sequence for $A$ with the same first step.  Moreover,
\[
  \nu_{Sg}(B)
  =\frac{\int_B\Psi(\lambda)\,\nu_g(d\lambda)}
         {\int\Psi(\lambda)\,\nu_g(d\lambda)},
\]
so the uniform bounds in \eqref{eq:hilb-weight-bounds} make $\nu_{Sg}$ and $\nu_g$ mutually absolutely continuous and hence give them the same topological support.  Bounded norm equivalence preserves root factors and transfers the nonexistence of envelopes below the threshold.  The inequalities
\[
  \sqrt{\psi_-}\norm{g}\le\norm{Sg}\le\sqrt{\psi_+}\norm{g}
\]
give \eqref{eq:hilb-weighted-prefactor-transfer}.  The matched weighted initialization is the ordinary matched BB1 initialization for $\widetilde g_0$, and equal endpoint energies for $\widetilde g_0$ are exactly \eqref{eq:hilb-weighted-balance}.  For data supported on the $a$- and $b$-eigenspaces, both original components are then multiplied in modulus by $(b-a)/(b+a)$ at every step, which gives \eqref{eq:hilb-exact-orbit}.
\end{proof}

\begin{proof}[Proof of \cref{cor:hilb-uniform-family}]
If $m=M$, every matched trajectory terminates after one update.  Assume $m<M$ and work on $X_{[m,M]}$.  Since every scalar spectral measure is supported in $[m,M]$, it suffices under the standing zero-extension convention to estimate nonterminal segments (singleton-supported matched trajectories terminate immediately).  Every such segment satisfies \cref{lem:hilb-semiuniform} with the same constant $B_{[m,M],\eta}$ for all $h$.  Choose $\eta>0$ so that $c_{m,M}e^\eta<\gamma$ and set
\[
  C=\max\{1,\exp(B_{[m,M],\eta})\}.
\]
The semi-uniform product bound proves \eqref{eq:uniform-family-envelope}.  For BB2, the conjugacy $S_h=A_h^{1/2}$ and
\[
  \sqrt m\norm{g}\le\norm{S_hg}\le\sqrt M\norm{g}
\]
give the common prefactor $\sqrt{M/m}\,C$.
\end{proof}

\end{document}